\documentclass[12pt]{amsart}
\usepackage{amsmath}
\usepackage{amssymb}
\usepackage{epsfig}
\usepackage{wasysym}
\usepackage{graphicx}
\usepackage{tikz}
\usepackage{tikz-cd}
\usepackage{bm}
\usepackage{xcolor}
\usepackage{listings}

\numberwithin{equation}{section}
\usepackage{extpfeil}
\usepackage{array}
\usepackage{adjustbox}
\usepackage{longtable}
\usepackage{makecell}
\usepackage[all]{xy}

\input xy
\xyoption{all}

\DeclareFontFamily{U}{mathb}{\hyphenchar\font45}
\DeclareFontShape{U}{mathb}{m}{n}{
      <5> <6> <7> <8> <9> <10> gen * mathb
      <10.95> mathb10 <12> <14.4> <17.28> <20.74> <24.88> mathb12
      }{}
\DeclareSymbolFont{mathb}{U}{mathb}{m}{n}
\DeclareMathSymbol{\righttoleftarrow}{3}{mathb}{"FD}

\calclayout
\allowdisplaybreaks[3]

\theoremstyle{plain}
\newtheorem{prop}{Proposition}[section]

\newtheorem{coro}[prop]{Corollary}

\newtheorem{lemm}[prop]{Lemma}

\theoremstyle{definition}

\newtheorem{exam}[prop]{Example}

\newcommand{\actsfromleft}{\mathrel{\reflectbox{$\righttoleftarrow$}}}
\newcommand{\actsfromright}{\righttoleftarrow}

\def\lra{\longrightarrow}

\def\cB{{\mathcal B}}

\def\BC{{\mathcal{BC}}}

\def\cO{{\mathcal O}}

\def\cS{{\mathcal S}}

\def\fA{{\mathfrak A}}

\def\fD{{\mathfrak D}}

\def\fK{{\mathfrak K}}

\def\fS{{\mathfrak S}}

\def\mp{{\mathfrak p}}

\def\fS{{\mathfrak S}}

\def\bP{{\mathbb P}}

\def\bZ{{\mathbb Z}}

\def\bN{{\mathbb N}}

\def\rH{{\mathrm H}}

\def\bF{{\mathbb F}}

\def\BC{\mathcal{BC}}

\def\Pic{\mathrm{Pic}}

\def\Aut{\mathrm{Aut}}

\def\SL{\mathrm{SL}}
\def\GL{\mathrm{GL}}

\def\PGL{\mathrm{PGL}}

\def\Hom{\mathrm{Hom}}

\def\Burn{\mathrm{Burn}}
\def\Bir{\mathrm{Bir}}

\def\lim{\mathrm{lim}}

\def\Ker{\mathrm{Ker}}

\makeatother
\makeatletter

\begin{document}

\title[Equivariant geometry of linear actions]{Equivariant birational geometry of linear actions}

\author{Yuri Tschinkel}
\address{
  Courant Institute,
  251 Mercer Street,
  New York, NY 10012, USA
}

\email{tschinkel@cims.nyu.edu}

\address{Simons Foundation\\
160 Fifth Avenue\\
New York, NY 10010\\
USA}

\author{Kaiqi Yang}
\address{
  Courant Institute,
  251 Mercer Street,
  New York, NY 10012, USA
}
\email{ky994@nyu.edu}

\author{Zhijia Zhang}
\email{zhijia.zhang@cims.nyu.edu}

\date{\today}

\begin{abstract}
We study linear actions of finite groups in small dimensions, up to equivariant birationality. 
\end{abstract}

\maketitle

\section{Introduction}
\label{sect:intro}

The classification of actions of finite groups on rational surfaces, up to equivariant birationality, has a rich past and an active present. It goes back at least to 
the classical work of Bertini, Castelnuovo, Kantor, Segre, with the focus on involutions and their fixed loci, to the work of Manin, Iskovskikh, and Sarkisov, with an emphasis on the group action on the Picard group, classification of elementary birational transformations, and equivariant birational rigidity. The fundamental work of 
Dolgachev--Iskovskikh \cite{DI} summarizes and completes this vast program, to a certain extent: 
it gives a list of finite groups that can act on rational surfaces, and presents an algorithm that allows to distinguish different birational actions of a group, in many cases. 

More precisely, the equivariant Minimal Model Program (MMP) shows that an action of a finite group $G$ on a rational surface can be realized as a regular action either on a Del Pezzo surface or conic bundle over $\bP^1$. One can assume that 
the surface is {\em minimal}, i.e., no equivariant blow downs are possible. 
Actions on minimal Del Pezzo surfaces of low degree are {\em rigid}, and visible via 
induced actions on the primitive Picard lattice, i.e., as subgroups of the respective Weyl group. 

The most significant {\em ``What is left?''} \cite[Section 9]{DI} was the classification, up to birationality, of actions on Del Pezzo surfaces of high degree, e.g., linear and projectively linear actions on the projective plane.  

Recall that {\em linear}, respectively, {\em projectively linear} actions of finite groups
$G$ arise via projectivizations $\bP(V)$ of an  $(n+1)$-dimensional representation $V$ of $G$, respectively, of a central extension of $G$. In classical terminology, these are called: 
\begin{itemize}
\item {\em intransitive}: if the representation $V$ is reducible,
\item {\em transitive but imprimitive}: if the action is not intransitive, but 
there is a nontrivial normal subgroup of $G$ acting intransitively;
\item{\em primitive}: neither of the above. 
\end{itemize}
The case of {\em primitive} actions was essentially settled, via equivariant MMP, in \cite{sako}. On the other extreme, the birational classification of linear 
actions of {\em abelian} groups has been settled, in all dimensions, in \cite[Theorem 7.1]{reichsteinyoussininvariant}.
In general, the classification of regular actions on $\bP^2$, up to birationality, is still an open problem.  

The case of threefolds is much more involved. As in dimension 2, the
birational classification of linear actions on $\bP^3$ is an open problem.  Significant progress has been achieved in analyzing {\em primitive} actions  \cite{CS-finite}, \cite{CSar}, or involutions in the Cremona group $\mathrm{Cr}_3$ (see \cite{Pro-inv}).

\

New equivariant birational invariants were defined in \cite{KPT} and \cite{BnG}. The definitions assume that the ground field is of characteristic zero and contains roots of unity of order dividing the order of $G$. 
The invariants are computed on an appropriate birational model $X$ (standard form) and 
take values in the {\em Burnside group} 
$$
\Burn_n(G),
$$ 
which is defined as a quotient of a {\em symbols group}
by explicit relations. The symbols encode information about loci with nontrivial 
abelian stabilizers, the weights of the induced action in the normal bundle to these loci, as well as the induced action on the corresponding function fields, see \cite{HKTsmall} for definitions and examples. The paper \cite{KT-vector} applied this 
formalism to the study of  actions on $\bP^2$ and produced new examples of non-birational {\em intransitive} actions. 

In this paper, we work over an algebraically closed field $k$ of characteristic zero. 
We apply the formalism of Burnside groups to the study of 
linear actions in dimensions $\le 3$. We make extensive use of the 
algorithm developed in \cite{KT-vector}, 
which allows to recursively compute the class in $\Burn_n(G)$ 
of a (projectively) linear action of a finite group $G$ on $\bP^n$. We have  
implemented this algorithm in {\tt magma} and compiled tables of classes of such 
actions on $\bP^2$ and $\bP^3$, see \cite{TYZ-web}. 
Among our results are: 
\begin{itemize}
\item In dimension 2, the Burnside formalism does not allow to distinguish 
primitive actions but does yield 
many new examples of non-birational linear and projectively linear actions.
\item In dimension 3, we exhibit new types of non-birational linear actions on $\bP^3$ as well as nonlinearizable actions on smooth quadrics. 
\end{itemize}
In essence, the Burnside formalism complements birational rigidity techniques as in \cite{sako}, \cite{CS-finite}, \cite{CSar}.

Here is the roadmap of the paper:
In Section~\ref{sect:genn} we recall basic facts concerning equivariant birational geometry and relevant classical invariants used to distinguish actions up to birationality. In Section~\ref{sect:gen}, we recall the definition of the Burnside group $\Burn_n(G)$ introduced in  \cite{KPT}; this group receives birational invariants of generically free actions of a finite group $G$ on $n$-dimensional varieties.
We tabulate the groups 
in small dimensions and for small $G$, and develop new tools for 
working with these groups. In Section~\ref{sect:com-class} we explain how to compute the class 
$$
[X\actsfromright G]\in \Burn_n(G)
$$ 
of a generically free $G$-action on an $n$-dimensional variety $X$. 
In Section~\ref{sect:dim1} we apply the formalism to curves. In Section~\ref{sect:com-lin} we give examples of computations of classes of linear actions, using the algorithm in \cite{KT-vector}. In Sections~\ref{sect:dim2} and \ref{sect:3} we investigate linear actions on $\bP^2$ and $\bP^3$, providing new examples of non-birational actions, not distinguishable with previous tools. 
In Section~\ref{sect:quad} we study smooth quadrics of dimension $\le 3$. 

\

\noindent
{\bf Acknowledgments:} 
We are very grateful to I. Cheltsov and A. Kresch for their interest and comments. 
The first author was partially supported by NSF grant 2000099.

\section{Generalities}
\label{sect:genn}

We recall basic terminology and notation. We consider {\em generically free}, regular actions of finite groups $G$ on smooth projective algebraic varieties over an algebraically closed field $k$ of characteristic zero. By convention, the action is from the right, and it will be denoted by 
$$
X\actsfromright G.
$$
The induced left $G$-action on the function field $K=k(X)$ is denoted by $G\actsfromleft K$. We let
$$
X^G:=\{ \mathfrak p \in X, \mathfrak p\cdot g = \mp\} 
$$
be the set of $G$-fixed points on $X$. 

We write 
$$
X\sim_G X',
$$
if there exists a $G$-equivariant birational map $X\dashrightarrow X'$. 
This means that there exists a $G$-equivariant isomorphism of field extensions
$$
k(X)/k \stackrel{\sim}{\lra} k(X')/k.
$$
We say that $X,X'$ are {\em stably} equivariantly birational if 
$$
X\times \bP^m \sim_G X'\times \bP^m, 
$$
for some $m$, with trivial action on the second factor. Of particular interest is the study of (conjugacy classes of) finite subgroups of the {\em Cremona group}
$$
\mathrm{Cr}_n=\Bir\Aut(\bP^n),
$$
the group of birational automorpisms of projective space, and  
the study of equivariant birationalities
$$
X\sim_G \bP(V).
$$
We say that the $G$-action on $X$ is: 
\begin{itemize}
\item {\em linearizable} if $V$ is a faithful representation of $G$, i.e., the action arises from an injective homomorphism $G\to \GL(V^\vee)$. 
\item {\em projectively linearizable} 
if the $G$-action on $\bP(V)$ arises from a {\em projective}
representation $G\to \PGL_{n+1}$, i.e., 
a linear representation $\tilde{G}\to \GL(V^\vee)$ of a central extension 
$$
1\to \mu_{n+1} \to \tilde{G} \to G\to 1.
$$
\end{itemize}
Note that a linearizable action is projectively linearizable, 
but the converse need not hold. We call the corresponding actions on $\bP(V)$ {\em linear}, respectively, {\em projectively linear}. 
Projectively linear actions on $\bP^n$ with a fixed point are linear. 

Among general approaches to the (stable) linearizability problem are:
\begin{itemize}
\item {\em birational rigidity}, see, e.g., \cite{Pro-ICM}, \cite{CS},
\item {\em intermediate Jacobians}, see \cite{HT-intersect}, 
\item group cohomology, such as {\em Amitsur invariant} (see see \cite[Section 6]{blancfinite}, \cite[Theorem~2.14]{sari}) 
or 
invariance of $\rH^1(G,\Pic(X))$ under equivariant blowups of smooth projective $G$-varieties $X$, see \cite{BogPro}. 
\end{itemize}

We list technical tools that are ubiquitous in 
equivariant birational geometry:
\begin{itemize}
\item If $X$ is rationally connected and $G$ is cyclic then $X^G\neq \emptyset$.
\item If $G$ is abelian and $\pi: \tilde{X}\dashrightarrow X$ is a $G$-equivariant birational map then 
$$
X^G\neq \emptyset \quad \Leftrightarrow \quad \tilde{X}^G\neq \emptyset.
$$
\item 
$\textbf{(RY)}$: Assume that a finite {\em abelian} group $G$ acts regularly and generically freely on a smooth projective variety $X$ of dimension $n$. Let $\mathfrak p\in X^G$ be a $G$-fixed point and 
$$
(a_1,\ldots,a_n), \quad a_j\in G^\vee
$$
the collection of characters of $G$ occurring in the tangent space at $\mathfrak p$. 
Let 
$$
\det(\mathfrak p): = a_1\wedge \cdots \wedge a_n\in \wedge^n(G^\vee)
$$
be the determinant. 
Let $\pi: \tilde{X}\to X$ be a $G$-equivariant birational morphism. Then, by
\cite{reichsteinyoussininvariant}, there exists a $G$-fixed point 
$\mathfrak q\in \pi^{-1}(\mathfrak p) \subset \tilde{X}$ such that
$$
\det(\mathfrak p)=\pm \det(\mathfrak q).
$$
\item 
{\bf (No-name lemma):}
If $G$ acts generically freely on $X$ and $\mathcal E\to X$ is a $G$-vector bundle of rank $m$  then 
$$
\mathcal E\sim_G X\times \bP^m,
$$
with trivial action on the second factor.
\item {\bf (MRC):}
Let $r=r(X)$ be the dimension of the Maximal Rationally Connected (MRC) quotient of an algebraic variety $X$. This is a well-defined equivariant birational invariant, by the functoriality of MRC quotients (see, e.g., \cite[IV.5.5]{Ko}). 
\item {\bf (H1):}
Let $X$ be a smooth projective variety with a generically free, stably linearizable, action of $G$. Then, for all $H\subseteq G$, one has
$$
\rH^1(H,\Pic(X)) =0.
$$
A $G$-variety satisfying this property will be called $\mathrm{H}^1$-trivial. This is a stable birational property.
\end{itemize}

In the next sections, we discuss $G$-birational invariants introduced in \cite{KPT} and \cite{BnG}. They are based on an analysis of the geometry of subvarieties of $X$ with nontrivial stabilizers, together with the induced representation in the normal bundle, and can be viewed as a generalization of the $\textbf{(RY)}$ invariant.

\section{Equivariant Burnside groups}
\label{sect:gen}

Throughout, $G$ is a finite group and $H$ a finite abelian group.
When $H\subseteq G$ is a subgroup, we write $Z_G(H)$, respectively $N_G(H)$, for its centralizer, respectively normalizer, in $G$. 
We write 
$$
H^\vee:=\Hom(H,k^\times)
$$ 
for the group of characters of $H$.

There are three versions of symbols groups, corresponding to the kind of data we attach to loci with nontrivial stabilizers (on a standard model, see Section~\ref{sect:com-class}).  
We recall the definitions, following
\cite{KPT} and \cite{BnG}.

\subsection{Maximal stabilizers}
\label{sect:maxstab}

This version addresses (generically free, regular) actions of abelian groups $H$ on smooth projective $X$, of dimension $n$; one records the weights of $H$ in the tangent space at $H$-fixed points.    
In detail, for $n\in \bN$, let
\[ \cS_n(H), \]
be the abelian group generated by \emph{symbols}
\[ \beta=(b_1,\dots,b_n), \quad b_1,\dots,b_n\in H^\vee, \quad 
\langle b_1,\dots,b_n\rangle=H^\vee,
\]
subject to the reordering relation \\
\\
\textbf{(O)} 
$\beta=(b_1,\dots,b_n)\sim \beta'=(b'_1,\dots,b'_n)$
if there is a permutation $\sigma\in \fS_n$,
with $b'_i=b_{\sigma(i)}$ for $i=1$, $\dots$, $n$.

\

Consider the quotient 
$$
\cS_n(H)\to \cB_n(H)
$$ 
by the blow-up relation\\
\\
\textbf{(B)} 
For $\beta=(b_1,\dots,b_n)$, $n\ge 2$,
\[ 
\beta=\begin{cases}
(0,b_2,\dots,b_n),&
\text{if $b_1=b_2$}, \\
\beta_1+\beta_2,&
\text{if $b_1\ne b_2$},
\end{cases}
\]
where
\[
\beta_1:=(b_1-b_2,b_2,b_3,\dots,b_n),\qquad
\beta_2:=(b_1,b_2-b_1,b_3,\dots,b_n).
\]

\

\subsection{Combinatorial Burnside group}

This version takes into account arbitrary stabilizers for actions of general finite groups, but ignores the induced action on function fields of strata with nontrivial stabilizers. For $n\in \bN$, let
\[ \mathcal{SC}_n(G) \]
be the abelian group generated by \emph{symbols}
\begin{equation}
\label{eqn:symbo}
(H,Y,\beta), 
\end{equation}
where
\begin{itemize}
\item 
$H\subseteq G$ is an abelian subgroup (the {\em stabilizer} of the symbol),
\item $Y$ is a subgroup of $Z_G(H)/H$,
and
\item $\beta=(b_1,\ldots, b_{n-d})$, with $d\in [0,\ldots, n]$, is a sequence of {\em nontrivial} characters of $H$, generating $H^\vee$.
\end{itemize}
Symbols with $d=0$ are called {\em point symbols} and those 
with $d=n-1$ {\em divisorial symbols}.

\

Symbols \eqref{eqn:symbo} are subject to reordering and conjugation relations: \\
\\
\textbf{(O)} $(H,Y,\beta)=(H,Y,\beta')$ if $\beta\sim \beta'$, as in
Section~\ref{sect:maxstab}.\\
\\
\textbf{(C)} For all $g\in G$, 
$$
(H,Y,\beta)=(H',Y',\beta'), \quad H'=gHg^{-1}, \quad 
Y'=gYg^{-1},
$$
and the characters in $\beta'$ arise from those in $\beta$ via conjugation by $g$.

\

Consider the quotient 
$$
\mathcal{SC}_n(G)\to \mathcal{BC}_n(G)
$$ 
by the vanishing and blowup relations:

\

\noindent
\textbf{(V)} $(H,Y,\beta)=0$ when $b_1+b_2=0$.\\
\\
\textbf{(B)} 
$
(H,Y,\beta)=\Theta_1+\Theta_2,
$
where:
\[ \Theta_1:=\begin{cases}
0,&
\text{if $b_1=b_2$}, \\
(H,Y,\beta_1)+(H,Y,\beta_2),&
\text{if $b_1\ne b_2$},
\end{cases}
\]
with $\beta_1$, $\beta_2$ as above, and
\[ \Theta_2:=\begin{cases}
0,&
\text{if $b_i\in \langle b_1-b_2\rangle$ for some $i$}, \\
(\overline{H},\overline{Y},\bar\beta),&
\text{otherwise}.
\end{cases}
\]
Here, 
$$
\overline{H}:=\Ker(b_1-b_2)\subseteq H,
$$
with 
$$
H/\overline{H}\subseteq \overline{Y}\subseteq Z_G(H)/\overline{H},
$$
characterized by $\overline{Y}/(H/\overline{H})=Y$,
and $\bar\beta$ consists of restrictions of characters of $\beta$:
\[ \bar\beta:=(\bar b_2,\bar b_3,\dots). \]
The images of point symbols, respectively, divisorial symbols, will be called {\em point classes}, respectively, {\em divisorial classes}.

\subsection{Equivariant Burnside group}
\label{sect:eqrel}

The most refined version records both the action of the stabilizer in the normal bundle and the induced action on the function fields of strata. 

\

For $n\in \bN$, let
\[ \mathrm{Symb}_n(G), \]
be the abelian group generated by symbols
\begin{equation}
\label{eqn:symb}
(H, Y\actsfromleft K, \beta),
\end{equation}
where
\begin{itemize}
\item $H\subseteq G$ is an abelian subgroup,
\item $Y\subseteq Z_G(H)/H$ is a subgroup,
\item $K$ is a finitely generated extension of $k$, of transcendence degree $d\le n$, with faithful action by $Y$, and
\item $\beta=(b_1,\ldots, b_{n-d})$ is a sequence of {\em nontrivial} characters of $H$, generating $H^\vee$. 
\end{itemize}
As in the case of combinatorial Burnside groups, we call a symbol in  $\mathrm{Symb}_n(G)$ {\em divisorial} if $d=(n-1)$, i.e., $\beta=(b)$, for some generator $b$ of $H^\vee$. We call a symbol a {\em point} symbol if $d=0$. 
Generally, we call $(n-d)$ the {\em codimension} of the symbol. 

\

Symbols \eqref{eqn:symb} 
are subject to reordering and conjugation relations: \\
\\
\textbf{(O)} $(H, Y\actsfromleft K, \beta)=(H, Y\actsfromleft K, \beta')$ if $\beta\sim \beta'$. \\
\\
\textbf{(C)} $(H, Y\actsfromleft K, \beta)=(H', Y'\actsfromleft K', \beta')$ if, for some $g\in G$, we have
$H'=gHg^{-1}$, $Y'=gYg^{-1}$, there is an isomorphism $K\cong K'$, trivial on $k$, that is compatible with the respective actions, and $\beta'$ obtained from $\beta$ via conjugation by $g$.

\

We consider the quotient
$$
\mathrm{Symb}_n(G) \to \Burn_n(G)
$$
by the vanishing and blowup relations:

\

\noindent
\textbf{(V)} $(H,Y\actsfromleft K,\beta)=0$ when $b_1+b_2=0$.\\
\\
\textbf{(B)} $(H,Y\actsfromleft K,\beta)=\Theta_1+\Theta_2$, where:
\begin{align*}
\Theta_1&:=\begin{cases}
0,&
\text{if $b_1=b_2$}, \\
(H,Y\actsfromleft K,\beta_1)+(H,Y\actsfromleft K,\beta_2),&
\text{if $b_1\ne b_2$},
\end{cases} \\
\Theta_2&:=\begin{cases}
0,&
\text{if $b_i\in \langle b_1-b_2\rangle$ for some $i$}, \\
(\overline{H},\overline{Y}\actsfromleft K(x),\bar\beta),&
\text{otherwise}.
\end{cases}
\end{align*}
Here $\overline{H} := \Ker(b_1-b_2)\subset H$ and $\bar{\beta}$ is the image of characters of $\beta$ in $\overline{H}^\vee$; 
there is also a recipe to produce a $\overline{Y}$-action on $K(x)$, extending the given action of $Y$ (via the canonical homomorphism $\overline{Y}\to Y$) on $K$, see the {\bf Action construction}  in \cite[Section 2]{BnG}.

\subsection{Computations} 
\label{sect:compute}

Let $G$ be abelian. The groups $\cB_n(G)$ are defined by finitely many generators and relations and are thus effectively computable. In practice, this is doable for $n\le 4$ and $|G|<300$. 
Such computations allowed to recognize 
interesting arithmetic and combinatorial structures of $\cB_n(G)$: these groups are related to cohomology of congruence subgroups of $\GL_n(\bZ)$, they carry Hecke operators, admit multiplication and comultiplication, see \cite{KPT}, \cite{KT-arith}, \cite{KT-struct}. 
Tables for cyclic groups $C_m$ of small order
can be found in \cite[Section 5]{KPT}.

The groups $\BC_n(G)$ are also finitely generated, with finitely many relations, and thus computable. 
A structure theorem, \cite[Theorem 5.2]{TYZ}, provides simplifications in computations of $\BC_n(G)$, by reduction to {\em modified} 
$\cB_n(H)$, for {\em abelian} subgroups $H\subseteq G$. For example, for $G$ abelian, we proved in \cite{TYZ} that 
$$
\BC_n(G)=\bigoplus_{H'\subseteq G} \bigoplus_{H''\subseteq H'} \, \cB_n(H''). 
$$
We list $\mathcal{B}_2,$ $\BC_2$ and $\BC_3$ for
small groups. We start with $G:=C_m$.

{\small

\renewcommand{\arraystretch}{1.1}
    \begin{longtable}{|l|l|l|l|}
    \hline
    &&&\\[-0.45cm]
    $m\!\!$&$\cB_2(G)\!\!$&$\mathcal{BC}_2(G)\!\!$&$\mathcal{BC}_3(G)\!\!$\\\hline
    &&&\\[-0.4cm]
$2\!\!$&$0\!\!$&$0\!\!$&$0\!\!$\\\hline
&&&\\[-0.4cm]
$3\!\!$&$\bZ\!\!$&$\bZ\!\!$&$0\!\!$\\\hline
&&&\\[-0.4cm]
$4\!\!$&$\bZ\!\!$&$\bZ$\!\!&$0\!\!$\\\hline
&&&\\[-0.4cm]
$5\!\!$&$\bZ^2\!\!$&$\bZ^2\!\!$&$0\!\!$\\\hline
&&&\\[-0.4cm]
$6\!\!$&$\bZ^2\oplus\bZ/2\!\!$&$\bZ^4\oplus \bZ/2\!\!$&$0\!\!$ \\\hline
&&&\\[-0.4cm]
$7\!\!$&$\bZ^3\oplus\bZ/2\!\!$&$\bZ^3\oplus\bZ/2\!\!$&$\bZ/2\!\!$\\\hline
&&&\\[-0.4cm]
$8\!\!$&$\bZ^3\oplus\bZ/4\!\!$&$\bZ^5\oplus\bZ/4\!\!$&$\bZ/2\!\!$\\\hline
&&&\\[-0.4cm]
$9\!\!$&$\bZ^5\oplus\bZ/3\!\!$&$\bZ^7\oplus\bZ/3\!\!$&$\bZ\!\!$\\\hline 
&&&\\[-0.4cm]
$10\!\!$&$\bZ^4\oplus(\bZ/2)^2\oplus\bZ/6\!\!$&$\bZ^8\oplus(\bZ/2)^2\oplus\bZ/6\!\!$&$(\bZ/2)^2\!\!$\\\hline
&&&\\[-0.4cm]
$11\!\!$&$\bZ^6\oplus\bZ/5\!\!$&$\bZ^6\oplus\bZ/5\!\!$&$\bZ\oplus\bZ/5\!\!$\\\hline
&&&\\[-0.4cm]
$12\!\!$&$\bZ^7\oplus\bZ/8\!\!$&$\bZ^{16}\oplus(\bZ/2)^2\oplus\bZ/8\!\!$&$\bZ^2\oplus(\bZ/2)^2\!\!$\\\hline
&&&\\[-0.4cm]
$13\!\!$&$\bZ^8\oplus\bZ/7\!\!$&$\bZ^8\oplus\bZ/7\!\!$&$\bZ^2\oplus\bZ/7\!\!$\\\hline 
&&&\\[-0.4cm]
$14\!\!$&$\bZ^7\oplus(\bZ/2)^4\oplus\bZ/12\!\!$&$\bZ^{13}\oplus(\bZ/2)^6\oplus\bZ/12\!\!$&$\bZ\oplus(\bZ/2)^{6}\!\!$\\\hline
&&&\\[-0.4cm]
$15\!\!$&$\bZ^{13}\oplus\bZ/8\!\!$&$\bZ^{19}\oplus\bZ/8\!\!$&$\bZ^5\oplus\bZ/2\!\!$\\\hline 
&&&\\[-0.4cm]
$16\!\!$&$\bZ^{10}\oplus(\bZ/2)^2\oplus\bZ/16\!\!$&$\bZ^{19}\oplus(\bZ/2)^2\oplus(\bZ/4)^2\oplus\bZ/16\!\!$&$\bZ^3\oplus(\bZ/2)^7\!\!$\\\hline
\end{longtable}

}

The next table concerns $G:=C_n\oplus C_m$. 

{\small

\renewcommand{\arraystretch}{1.1}
\begin{longtable}{|l|l|l|l|}
   \hline
    &&&\\[-0.45cm]
    $(n,m)\!\!$&$\cB_2(G)\!\!$&$\mathcal{BC}_2(G)\!\!$&$\mathcal{BC}_3(G)\!\!$\\\hline
$(2,2)\!\!$&$(\bZ/2)^2\!\!$&$(\bZ/2)^2\!\!$&$0\!\!$\\\hline
&&&\\[-0.4cm]
$(2,4)\!\!$&$\bZ^2\oplus(\bZ/2)^3\!\!$&$\bZ^6\oplus(\bZ/2)^7\!\!$&$(\bZ/2)^3\!\!$\\\hline
&&&\\[-0.4cm]
$(2,6)\!\!$&$\bZ^3\oplus(\bZ/2)^4\oplus\bZ/4\!\!$&$\bZ^{20}\oplus(\bZ/2)^{14}\oplus\bZ/4\!\!$&$(\bZ/2)^{9}\!\!$\\\hline
&&&\\[-0.4cm]
$(2,8)\!\!$&$\bZ^6\oplus(\bZ/2)^6\oplus\bZ/8\!\!$&$\bZ^{30}\oplus(\bZ/2)^{18}\oplus(\bZ/4)^4\oplus\bZ/8\!\!$&$\bZ\oplus(\bZ/2)^{24}\!\!$\\\hline 
&&&\\[-0.4cm]
$(4,4)\!\!$&$\bZ^{11}\oplus\bZ/2\!\!$&$\bZ^{41}\oplus(\bZ/2)^{29}\!\!$&$\bZ^5\oplus(\bZ/2)^{31}\!\!$\\\hline
&&&\\[-0.4cm]
$(3,3)\!\!$&$\bZ^7\!\!$&$\bZ^{15}\!\!$&$\bZ^3\!\!$\\ 
\hline
\end{longtable}
}

We also record results for small nonabelian $G$.

\renewcommand{\arraystretch}{1.1}
\begin{longtable}{|c|c|c|}
\hline
$G$ & $\BC_2(G)$ & $\BC_3(G)$\\
\hline
$Q_8$ & $(\bZ/2)^3$ & 0\\
\hline
$\fD_4$ & $(\bZ/2)^3$ & $0$\\
\hline
$\fD_5$ & $(\bZ/2)^2$ & 0\\
\hline
$\fA_5$ & $(\bZ/2)^3$ & 0 \\
\hline
$\fS_5$ & $(\bZ/2)^6 \oplus \bZ/4$ & 0 \\
\hline
$\fD_6$ & $(\bZ/2)^5 \oplus \bZ/4$ & 0\\
\hline
$\fA_6$ & $(\bZ/2)^7 \oplus \bZ/4 \oplus \bZ$ & $\bZ/2\oplus \bZ$ \\
\hline
$\fS_6$ & $(\bZ/2)^{31} \oplus (\bZ/4)^3 \oplus \bZ/8$ & $(\bZ/2)^5 \oplus \bZ/4$\\
\hline
$\fA_7$ & $(\bZ/2)^{12} \oplus (\bZ/4)^3 \oplus \bZ/8 \oplus \bZ^2$ & $(\bZ/2)^3\oplus \bZ$\\
\hline
$\mathrm{PSL}_2(\mathbb F_7)$ & $(\bZ/2)^3 \oplus \bZ$ & $\bZ/2$\\
\hline
$\fD_{5} \times \fD_4$ & $(\bZ/2)^{118} \oplus \bZ/4 \oplus (\bZ/12)^{11} \oplus (\bZ/24) \oplus \bZ$ & $(\bZ/2)^{63}\oplus \bZ$\\
\hline
\end{longtable}

In contrast to $\cB_n(G)$ and $\BC_n(G)$, the computation of $\Burn_n(G)$ is more difficult.
One of the reasons is that the symbols depend on function fields, i.e., algebraic varieties, which have {\em moduli}. For example, there are 3 types of nonlinearizable involutions in the plane Cremona group $\mathrm{Cr}_2$ (de Jonqui\`eres, Geisser, Bertini), fixing curves $C$ of genus $\ge 1$, and contributing symbols
$$
\mathfrak s=(C_2,1\actsfromleft k(C), (1)) \in \Burn_2(C_2). 
$$
Since the conjugacy class of an involution in $\mathrm{Cr}_2$ is uniquely determined by $k(C)$, the symbols $\mathfrak s$ parametrize all conjugacy classes of involutions. 

In the following sections, we will discuss various approaches to working with $\Burn_n(G)$. There is a natural homomorphism 
\begin{align}\label{eq:BurntoBC}
\Burn_n(G)\to \BC_n(G),    
\end{align}
defined by forgetting the field information in each symbol (see \cite[Section 8]{KT-struct}). Note that it is {\em not} necessarily surjective. However, sometimes, this homomorphism allows to 
distinguish actions by comparing their classes 
under the homomorphism \eqref{eq:BurntoBC}, see Section \ref{sect:dim2}, \ref{sect:3} and \ref{sect:quad}.

\subsection{Tools}
\label{sect:further}

In small dimensions and for small $G$, we can arrive at simplifications 
via simple manipulations with defining relations. For reference, we list several such standard
operations with symbols, which are independent of the ambient group and will be frequently used.  

We consider symbols
\begin{equation}
\label{eqn:11}
\mathfrak s=(H, Y\actsfromleft K, \beta), \quad  
\beta=(b_1,\ldots, b_{n-d}), \quad K=k(F), 
\end{equation}
with small $H$ and $Y$.

\

\noindent
{\em Reduction to point classes:} Relation $\mathbf{(B)}$
implies that if $d\neq n-1$ and
$b_1=b_2$ then
\begin{equation}
\label{eqn:point}
\mathfrak s=
(H,Y\actsfromleft K(x), (b_2,\ldots,b_{n-d})),
\end{equation}
with trivial $Y$-action on $x$. In particular, every symbol as in \eqref{eqn:11} with 
$Y=1$ and $F=\bP^d$ can be reduced to a point symbol.

\

\noindent
{\em Vanishing:} 
Relation $\mathbf{(V)}$ implies that $\mathfrak s$
vanishes, provided
\begin{equation}
\label{eqn:vv}
\sum_{i\in I} b_i = 0\in H^\vee, \text{ for some } I\subseteq [1,\ldots, n-d].
\end{equation}

\

\noindent
{\em Cyclic stabilizers:} 

\begin{itemize}
\item $H=C_2$: 
If $\beta$ contains more than one entry, 
$\mathfrak s=0\in\Burn_n(G)$, by $\mathbf{(V)}$. 
Assume that
$$
F\sim_Y F'\times \bP^1,
$$
with trivial action on the second factor. 
By \eqref{eqn:point} and $\mathbf{(V)}$, 
$$
\mathfrak s=(C_2, Y\actsfromleft k(F'), (1,1))=0.
$$ 
\item $H=C_3$: The symbol $\mathfrak s$ vanishes, if its codimension is $\ge 3$, by \eqref{eqn:vv}. Together with $\mathbf{(B)}$ this implies
$$
(C_3, 1\actsfromleft K, \beta) = 0 \in \Burn_n, \quad \text{ for } n\ge 3.
$$
For some $G$, the symbol can be nontrivial, i.e., in $\Burn_2(C_3)$.
On the other hand, if there is a $C_6\subset G$ centralizing $H$, then it supplies additional relations, leading to additional vanishing. For example, we have 
\begin{align*}(C_3,C_2\actsfromleft k(\bP^1),(1,1))=&(C_6,1\actsfromleft k,(1,4,1))-(C_6,1\actsfromleft k,(1,3,1))\\&-(C_6,1\actsfromleft k,(3,4,1))\\
=&-(C_6,1\actsfromleft k,(1,3,1))\!-\!(C_6,1\actsfromleft k,(5,4,1))\\&-(C_6,1\actsfromleft k,(3,1,1))\\
=&-2(C_6,1\actsfromleft k,(1,3,3))=0\in\Burn_3(G).\end{align*}
Similarly, 
\begin{align*}(C_3,C_2\actsfromleft k(\bP^1),(2,2))=&(C_6,1\actsfromleft k,(2,5,5))-(C_6,1\actsfromleft k,(2,3,5))\\&\!-(C_6,1\actsfromleft k,(3,5,5))\\
=&\!-(C_6,1\actsfromleft k,(2,1,5))-(C_6,1\actsfromleft k,(5,3,5))\\&\! -(C_6,1\actsfromleft k,(3,5,5))=0\\
=&\!-2(C_6,1\actsfromleft k,(5,3,3))=0\in\Burn_3(G).
\end{align*}
\item $H:=C_4$: Consider point symbols for $n=3$. 
There are only two potentially nontrivial symbols 
\begin{equation}
\label{eqn:c4}
(C_4, 1\actsfromleft k, (1,1,1)), \quad (C_4, 1\actsfromleft k,(3,3,3)), 
\end{equation}
Using \eqref{eqn:vv}, we derive
$$
0=(C_4,1\actsfromleft k, (1,2,1))= 
(C_4,1\actsfromleft k, (3,2,1)) + (C_4,1\actsfromleft k, (1,1,1)),
$$
and thus the second term on the right vanishes. 
The same argument applies to the other symbol in 
\eqref{eqn:c4}.
\item $H=C_5$: All symbols
$$
(C_5, 1\actsfromleft k(\bP^d),\beta)\in \Burn_n(G), \quad n\ge 2,  
$$
reduce to point classes. 
Let $n=3$ and order $b_1\le b_2\le b_3$, using $\textbf{(O)}$. 
Potentially nonvanishing generators are:
$$
(C_5, 1\actsfromleft k, (i,i,i)), i=1,\ldots, 4, \quad 
(C_5, 1\actsfromleft k, (1,1,2)),
$$
and turn to relations:
$$
(C_5, 1\actsfromleft k, (1,1,2))=
(C_5, 1\actsfromleft k, (1,4,2))+
(C_5, 1\actsfromleft k, (1,1,1)).
$$
On the other hand, we have
$$
(C_5, 1\actsfromleft k, (1,1,2)) = 
(C_5, 1\actsfromleft k(\bP^1), (1,2)) = 
(C_5, 1\actsfromleft k, (1,2,2))\!=\!0.
$$
The same argument shows the vanishing of all other generators. 
\end{itemize}

To summarize, we have:

\begin{lemm}
\label{lemm:triv-point} 
Let $G$ be a finite group and $n\ge 3$.
Every point class in 
$\Burn_n(G)$,  
with stabilizer $H=C_m\subset G$ and 
$m\le 6$ is trivial. 
\end{lemm}

\begin{proof}
It suffices to prove this for $n=3$. 
We already dealt with $m=2,3,4,5$. 
When $m=6$, $\Theta_2$-terms in the blow-up relations come from:
$$
(C_2,C_3\actsfromleft k(\bP^1),(1,1))=0,
$$
 $$(C_3,C_2\actsfromleft k(\bP^1),(1,2))=0,
$$
$$
 (C_3,C_2\actsfromleft k(\bP^1),(\pm 1,\pm 1)).
$$
We prove that the last symbols are also zero in $\Burn_3(G)$. 
First of all, 
\begin{align*}
0& =(C_3,C_2\actsfromleft k(\bP^1),(1,2)) \\
& = 
(C_3,C_2\actsfromleft k(\bP^1),(2,2))+ 
(C_3,C_2\actsfromleft k(\bP^1),(1,1)).
\end{align*}
For compactness, for point classes, we will use the notation
$$
(b_1,b_2,b_3)=(C_6, 1\actsfromleft k(\bP^1),(b_1,b_2, b_3)).
$$
Applying $\mathbf{(B)}$, we obtain
$$
0=(1,4,1)=(3, 4, 1) + (1, 3, 1)  + (C_3,C_2\actsfromleft k(\bP^1),(1,1)).
$$ 
Similarly, 
$$
(3, 4, 1)= (3, 4, 3) + (3, 3, 1) +\Theta_2=0, 
$$
since all terms on the right vanish, by $\mathbf{(V)}$ and the fact that $$
b_3 \in \langle b_1-b_2\rangle, \quad b_3=3, b_1=1, b_2=4.
$$
We now have
\begin{equation}
\label{eqn:341}
0=(3, 4, 1)=(5,4,1)+(3,1,1).  
\end{equation}
Thus, all $\Theta_2$ terms vanish. 

Next, note that once we know that $(b_1,b_2,b_3)=0$ then the same relations, applied to negatives, yield $(-b_1,-b_2,-b_3)=0$ as well. 
Thus we need to prove the vanishing of the non-boldface symbols  in the following sequence of relations, which we apply in the given sequence; in bold we have
indicated the terms that vanish by $\textbf{(V)}$, by previous identities, or by sign change on previously obtained vanishing symbols:

\begin{align*}
\mathbf{(1,2,3)} & = \mathbf{(1,3,5)}  + (1,1,2) \\
\mathbf{(1,1,2)} & = \mathbf{(1,5,2)} + (1,1,1)  \\
\mathbf{(1,2,3)} & = \mathbf{(4,2,3)} + (1,2,2)  \\
(1,3,4) & = \mathbf{(3,3,4)} + \mathbf{(1,3,3)}   \\
(1,1,3) & = \mathbf{(1,4,3)}  + \mathbf{(1,1,2)}   \\
(2,2,3)  & = \mathbf{(2,5,3)}  + \mathbf{(2,2,1)}   \\
(1,4,4) & = \mathbf{(3,4,4)}  + \mathbf{(1,3,4)}   
\end{align*}
\end{proof}

\subsection{Incompressibles}
\label{sect:inc}

For $n=1$, there are no relations, with the exception of the conjugation relation $({\bf C})$, i.e., $\Burn_1(G)$ is the free abelian group
spanned by symbols
$$
(H, 1\actsfromleft k, (b_1)),
$$
where $H\subseteq G$ is a {\em cyclic} subgroup (up to conjugation). 

In dimensions $n\ge 2$, we call a divisorial symbol {\em incompressible} 
if it {\em does not} appear in the $\Theta_2$-term of 
any relation $\mathbf{(B)}$. 
We have 
\begin{equation}
\label{eqn:deco}
\Burn_n(G) = \Burn_n^{\rm triv}(G) \oplus \Burn_n^{\rm inc}(G) \oplus  \Burn_n^{\rm comp}(G),
\end{equation}
where 
\begin{itemize}
\item $\Burn_n^{\rm triv}(G)$ is freely spanned by symbols
$$
(1,G\actsfromleft K, ()),
$$
where $K$ is a field of transcendence degree $n$, with a generically free action of $G$; 
\item 
$\Burn_n^{\rm inc}(G)$ is freely spanned by {\em incompressible divisorial symbols}, modulo conjugation; and 
\item 
the third summand is generated by all other symbols, 
subject to relations in Section~\ref{sect:eqrel} (see \cite[Proposition 3.4]{KT-vector}).
\end{itemize}
In some examples, the presence of incompressible symbols already allows to distinguish birational types of actions, greatly simplifying the arguments (see Section~\ref{sect:dim2}).  In other examples, one has to perform computations in 
$\Burn_n^{\rm comp}(G)$.

Recall that, for $n=2$, we have
\begin{itemize}
\item {\em point} classes, i.e., $K=k$ and $\beta=(b_1,b_2)$,
\item {\em divisorial classes}:
\begin{itemize}
\item classes of rational curves, i.e., $K=k(x)$, $\beta=(b_1)$, and $Y$ cyclic,
\item classes of rational curves, with $\beta=(b_1)$, and $Y$ noncyclic,
\item classes of curves of genus $\ge 1$, i.e., those where $K=k(C)$, and $C$ is a
curve of genus $\ge 1$. 
\end{itemize}
\end{itemize}
The {\em incompressible} divisorial symbols correspond to the last two cases.

The table below  shows the structure of 
$\Burn_2^{\rm comp}(G)$ for $G=C_m$:   

\renewcommand{\arraystretch}{1.1}
\begin{longtable}{|l|l|}
\hline
$m$ & $\Burn_2^{\rm comp}(G)$ \\
\hline
 2 & $0$\\
\hline
 3 & $\bZ$\\
\hline
 4 & $\bZ^2$\\
 \hline 
5 & $\bZ^2$ \\
\hline
6 & $\bZ^6$ \\
\hline
7 & $\bZ^3 \oplus \bZ/2$ \\
\hline
8 & $\bZ^8 \oplus \bZ/2$\\
 \hline
9 & $\bZ^8 \oplus \bZ/3$ \\
\hline
10 & $\bZ^{11} \oplus \bZ/3$\\
\hline
11 & $\bZ^6 \oplus \bZ/5$\\
\hline
12 & $\bZ^{22} \oplus \bZ/4$ \\
\hline
13 & $\bZ^8 \oplus\bZ/7$\\
 \hline 
14 & $\bZ^{17} \oplus \bZ/2 \oplus \bZ/6$\\
 \hline 
15 & $\bZ^{22} \oplus \bZ/8$\\
 \hline
16 & $\bZ^{25} \oplus(\bZ/2)^2 \oplus \bZ/8$
\\
\hline
\end{longtable}


The analysis of incompressible divisorial symbols
$$
\bar{\mathfrak s}=(\bar{H}, \bar{Y}\actsfromleft k(D), (\bar{b}_1)), 
$$
in dimensions  $n\ge 3$ is more involved. We have not attempted a full classification, but can identify several types, e.g., 
\begin{itemize}
\item $D$ is not uniruled,
\item $D$ is $G$-solid, i.e., not $G$-birational to a $G$-equivariant Mori fiber space over a positive-dimensional base (see \cite{CDK} for a detailed study of  {\em toric} $G$-solid varieties in dimension $\le 3$),
\item $n=3$ and $D$ is a rational surface which is not $\bar{Y}$-equivariantly birational to a Hirzebruch surface, see \cite{DI} for a 
classification of such actions.
\end{itemize}

How to tell whether or not a symbol 
\begin{equation}
\label{eqn:sbar}
\bar{\mathfrak s}:=(\bar{H}, \bar{Y}\actsfromleft \bar{K}, \beta) \in \Burn_n(G)
\end{equation}
is incompressible, in practice? A {\em necessary} condition is that 
$$
\bar{K}\not\sim_{\bar{Y}} \bar{K'}(x), 
$$
for some function field $\bar{K}'$, with trivial action of $\bar{Y}$ on $x$; such symbols arise via blowup relations from symbols where some characters in $\beta$  
have multiplicity $\ge 2$. The next steps, after verifying this condition, are: 
\begin{enumerate}
\item List all conjugacy classes of abelian subgroups $H\subseteq G$, together with their centralizers $Z_G(H)$.
\item 
For each $H$ enumerate all nontrivial proper subgroups $H'\subsetneq H$. List all subgroups
$$
Y'\subseteq Z_G(H)/H'.
$$
\item If there is no $(H',Y')$ conjugated to $(\bar{H}, \bar{Y})$ then $\bar{\mathfrak s}$ is incompressible. 
\item 
If there is such a pair, one needs to analyze in detail 
whether or not the Action construction can produce, birationally, the given action $\bar{Y}\actsfromleft \bar{K}$.
\end{enumerate}

\begin{exam}
\label{exam:incomp}
Let $n=3$ and $G=Q_8$. There are 4 conjugacy classes of nontrivial abelian subgroups, one $C_2$, with centralizer $G$, and three $C_4$, with centralizer itself. 
We consider the {\em divisorial} symbol
$$
\bar{\mathfrak s} = (C_2, C_2\times C_2\actsfromleft \bar{K}, (1))\in \Burn_3(G),  
$$
where $\bar{K}=k(\bP^2)$ and $\bar{Y} = C_2\times C_2$ 
acts {\em linearly}, in particular, with fixed points.
Such an action is not birational to an action of $\bar{Y}$ on $\bP^1\times \bP^1$, with trivial action on the second factor. 
By Step 2 above, such $\bar{Y}=C_2\times C_2$ do 
not arise. 
\end{exam}

\subsection{MRC quotients}
\label{sect:mrc}

Another look at the key relation 
$\textbf{(B)}$ shows that 
the function field $K=k(F)$ in the symbol
$$
\mathfrak s =(H, Y\actsfromleft k(F), \beta)
$$
on the left side is the function field of a rationally connected (RC) variety iff this holds for $K(x)$ in the $\Theta_2$-term on the right side. In fact, in any given relation, all appearing terms have the
same dimension of the MRC quotient $r=r(F)$. 
This yields a direct sum decomposition
\begin{equation}
\label{eqn:rc}
\Burn_n(G) =  \Burn_n^{\rm triv}(G) \oplus \Burn_n^{\rm rc}(G) \oplus \bigoplus_{r=1}^{n-1}\Burn_n^{\rm nrc,r}(G),
\end{equation} 
where 
\begin{itemize}
\item 
$\Burn_n^{\rm triv}(G)$ is freely
spanned by symbols with $H=1$, 
\item 
$\Burn_n^{\rm rc}(G)$ is generated by symbols $\mathfrak s$ with $H\neq 1$, and fields $K=k(F)$, where 
$F$ is a rationally connected variety, and
\item $\Burn_n^{\rm nrc,r}(G)$ is generated by symbols with $H\neq 1$ and 
$K=k(F)$ the function field of a variety
whose MRC quotient has 
dimension $r$. 
\end{itemize}
Different summands in this decomposition could have nontrivial intersection with $\Burn_n^{\rm inc}(G)$, the incompressible divisorial symbols.

\subsection{$\mathrm{H}^1$-triviality}
\label{sect:h1-triv}
Further decompositions of $\Burn_n(G)$ can be obtained by realizing that relation $\mathbf{(B)}$ preserves
$$
\rH^1(Y',F), \quad Y'\subseteq Y, 
$$
where $F$ is a smooth projective model of the function field in the symbol $\mathfrak s$. 
In particular, we have
$$
\Burn_n^{\mathrm{rc}}(G) = \Burn_n^{\mathrm{rc},\mathrm{H1}=0}(G) \, \oplus \, 
\Burn_n^{\mathrm{rc},\mathrm{H1}\neq 0}(G),
$$
depending on the (non)triviality of the $\rH^1$-condition (see Section~\ref{sect:gen}). 

\begin{lemm}
\label{lemm:h1}
If $\bar{\mathfrak s}\in \Burn_3^{\mathrm{rc}}(G)$ is a compressible divisorial symbol then 
$$
\bar{\mathfrak s} \in \Burn_n^{\mathrm{rc},\mathrm{H1}=0}(G).
$$
\end{lemm}

\begin{proof}
Indeed, it can only arise from a symbol
$$
\mathfrak s = (H, Y\actsfromleft k(\bP^1),\beta)
$$
which is $\rH^1$-trivial. 
\end{proof}

\section{Computing the classes}
\label{sect:com-class}

We recall the definition of the class of a generically free $G$-action on a smooth projective variety $X$. By convention, the $G$-action on $X$ is on the right, and the induced action on $K=k(X)$ is on the left. 

We assume that $X$ is in {\em standard form}, i.e., there is an open subset $U\subset X$ where the $G$-action is free, with complement $X\setminus U$ 
a normal crossings divisor such that for its every component $D$ and all $g\in G$, we have
$(D\cdot g) \cap D$ is either empty or all of $D$, see \cite[Section 7.2]{HKTsmall} for more details. Such a model of the function field $K$ can always be obtained via equivariant blowups, and every further blowup of such a model is also in standard form. One of its features is that all stabilizers are {\em abelian}. 
By definition, the class of such an action
\begin{equation}
\label{eqn:class}
[X\actsfromright G] :=\sum_H \, \sum_F (H, Y\actsfromleft k(F), \beta_F) \in \Burn_n(G)
\end{equation}
is a sum over conjugacy classes of stabilizers $H$ of maximal strata $F$ 
with these stabilizers, with the induced action of a subgroup $Y\subset Z_G(H)/H$ on the corresponding function field. In other words, the symbol records {\em one} representative of a $G$-orbit of a (maximal) stratum with stabilizer $H$: changing a component in this $G$-orbit conjugates the stabilizer by an element $g\in G$, the action on that component, and the induced action in the normal bundle to that component; this is reflected in the conjugation relation $\mathbf{(C)}$.  

The sum \eqref{eqn:class} contains a distinguished summand, 
$$
(1, G\actsfromleft k(X), ()) \in \Burn_n^{\mathrm{triv}}(G)
$$
reflecting the $G$-action on the generic point of $X$. 
Of course, there can be actions where there are no other summands in \eqref{eqn:class}, e.g., a translation action on an elliptic curve. In such cases, the Burnside group formalism provides {\em no} information about the $G$-action. On the other hand, we will exhibit many examples, where the actions can be distinguished via images of the corresponding classes under projections to 
$\Burn_n^{\mathrm{inc}}(G)$ or $\Burn_n^{\mathrm{comp}}(G)$.

We note that incompressible divisorial symbols can be read off from {\em any} equivariant birational model, even one which is not in standard form. 
It is typically a nontrivial task to find a standard model. Indeed, a linear representation $V$ of a nonabelian group $G$, and its equivariant compactification 
$\bP(1\oplus V)$, where $1$ is the trivial representation, by definition have strata with 
nonabelian stabilizers, e.g., the origin of $V$; and one may have to blow up several times to reach abelian stabilizers. 
In \cite{KT-vector} it was shown that a $G$-equivariant version of De Concini--Procesi compactifications of subspace arrangements provides a standard model for the $G$-action on $\bP(V)$; here the relevant subspaces in $\bP(V)$ 
correspond to loci with nontrivial stabilizers. We illustrate this in Section~\ref{sect:com-lin}. 
A similar algorithm for actions on toric varieties was presented in \cite{Burntoric}.

Next, assume we are given different $G$-actions, presented on $X$ and $X'$, which are both in standard form. 
To distinguish these, one expresses the classes
as in \eqref{eqn:class}, and 
considers the projection of the difference
$$
[X\actsfromright G]-[X'\actsfromright G]
$$
to 
$$
\Burn_n^{\mathrm{inc}}(G).
$$
Since there are no blowup relations between symbols in that group, it is easy to see whether or not this difference vanishes; 
see Corollary~\ref{coro:incbir}. 

If the difference does vanish in this group, we can consider projections to 
other direct summands introduced in Sections~\ref{sect:inc}, 
\ref{sect:mrc}, and \ref{sect:h1-triv} 
$$
\Burn_n^{\mathrm{comp}}(G), \quad \Burn_n^{\mathrm{rc}}(G), \ldots
$$
As mentioned in Section~\ref{sect:gen}, these groups are harder to compute, in general. One of the main difficulties is that one has to keep track of infinitely many generating symbols, and of relations that are implied by 
(often nontrivial) stable equivariant birationalities. For example, by the No-name Lemma, any two faithful $G$-representations are stably equivariantly birational, but not necessarily equivariantly birational. Further examples of such stable equivariant birationalities can be found in \cite{HT-torsor}. In some cases, we are able to overcome this intrinsic difficulty by passing to the combinatorial Burnside group $\BC_n(G)$, via \eqref{eq:BurntoBC}. 
We have implemented algorithms checking nonvanishing of any given class in $\BC_n(G)$, for all $n\ge 2$; however, these are practical only for small $n$.   

In the following sections, we will apply this machinery to 
\begin{itemize}
\item (projectively) linear actions on $\bP^n$, with $G\subset \PGL_{n+1}$, $n\le 3$,
\item smooth quadric hypersurfaces
$X\subset \mathbb{P}^n$, $n\le 4$. 
\end{itemize}

\section{Linear actions in dimension one}
\label{sect:dim1}

We recall the well-known list of finite 
$G\subset\PGL_2$:
$$
   C_m, \mathfrak{D}_{m}, \mathfrak{A}_4,\mathfrak{S}_4,\mathfrak{A}_5,
$$
where $C_m$ is the cyclic group of order $m$ and $\fD_m$ is the dihedral group of order $2m$. 
The corresponding actions on $\bP^1$ are linear if and only if 
$G$ is cyclic, or dihedral with $m$ odd. 

 The classification of birational actions on $\bP^1$ is straightforward: two $G$-actions on $\bP^1$ are equivariantly birational if and only if the corresponding representations $V$ are projectively equivalent, i.e., conjugated in $\PGL_2$. In detail:
\begin{itemize}
    \item $G=C_m$: the action arises via
      a representations of the form $\bP(1 \oplus \epsilon)$, where $\epsilon$ is a primitive character of $G$; given $\epsilon, \epsilon'$, birationality of the corresponding $G$-actions holds
     if and only if 
     $
     \epsilon=\pm\epsilon'.
     $
    \item $G=\fD_m$: when $m$ is odd, $G$ acts on $\bP^1$ via a faithful two-dimensional representations of $\fD_m$; when $m$ is even, $G$ acts via a faithful two-dimensional representations of $\fD_{2m}$. Two such actions are birational if and only if their restrictions to the subgroup $C_m\subseteq\fD_m$ induce birational actions of $C_m$ on $\bP^1$.
    \item $G=\fA_4$: the actions arise from faithful two-dimensional representations of $\SL_2(\mathbb F_3)$, all of which are projectively equivalent. So $\fA_4$ admits a unique action on $\bP^1$.
    \item $G=\fS_4$: the actions arise from faithful two dimensional representations of $\GL_2(\mathbb F_3)$, all of which are projectively equivalent. So $\fS_4$ also admits a unique action on $\bP^1$.
    \item $G=\fA_5$: the actions arise from faithful two-dimensional representations of $\SL_2(\mathbb F_5)$. There are two such representations, inducing two non-isomorphic actions of $\fA_5$ on $\bP^1$ after projectivization. So $\fA_5$ admits two non-birational actions.  
\end{itemize}
Note that in dimension 1, non-birational actions of cyclic groups can be distinguished by the Reichstein-Youssin invariant $\textbf{(RY)}$ \cite{reichsteinyoussininvariant}:
when $C_m$ acts on $\bP^1$ via a character $\chi$, 
the action is determined by $\pm \chi$. 

In applications to nonabelian groups, we can consider determinants of actions upon restrictions to their abelian subgroups, e.g., for $G$ dihedral. For $G=\fA_5$, the two non-birational actions can also be distinguished already via restriction to $C_5\subset G$: in one case the weights at the fixed points are $(1)$ and $(4)$ and in the other case $(2)$ and $(3)$.

\begin{prop}
\label{bir-p1}
The birational type of the action of a finite group $G$ on $\bP^1$ is uniquely determined by 
$$
[\bP^1\actsfromright G] \in \Burn_1(G). 
$$
\end{prop}

\section{Computing the classes of linear actions}
\label{sect:com-lin}

The computation of classes in the Burnside group of (projectively) linear actions in dimensions $\ge 2$ is more involved. 
Given a faithful linear representation $G\to\GL(V^\vee)$
 we obtain a faithful projective representation $G/C\to\PGL(V^\vee)$, 
 where $C\subset G$ is 
the maximal (cyclic) subgroup acting via scalar matrices. An algorithm to compute the class
$$
[\bP(V) \actsfromright G/C] \in \Burn_n(G/C)
$$
of the induced action of $G/C$ on $\bP^n$ was developed in \cite{KT-vector}, and implemented in \cite{TYZ-web}.
It is based on an equivariant version of the De Concini--Procesi approach to wonderful compactifications of subspace arrangements, which provides a systematic way of turning any given projectively linear action into a standard form. 
We note that 
\begin{itemize}
\item all symbols produced and appearing as summands in 
$$
[\bP(V)\actsfromright G]=\sum_{H} \, \sum_F (H, Y\actsfromleft k(F), \beta_F), 
$$
are in 
$$
\Burn_n^{\mathrm{rc}}(G),
$$
see \eqref{eqn:rc}, and 
\item 
all actions $Y\actsfromleft k(F)$ are 
equivariantly birational to products of
projectively linear actions on projective spaces, without
permutation of the factors (see Corollary~\ref{coro:yacts}). 
\end{itemize}

We explain the main ideas below, 
supplemented with two examples (our notation follows the one in \cite{KT-vector}). 
First, consider pairs 
\begin{equation}
\label{eqn:gamma}
(\Gamma, \epsilon),\quad C\subseteq\Gamma\subseteq G,\quad \epsilon\in\Hom(\Gamma, k^\times),
\end{equation}
where $\Gamma$ is the generic stabilizer group of some one-dimensional subspace $\ell\subset V$ and $\epsilon$ is the character of $\Gamma$ given by its action on $\ell$. Then $\Gamma/C$ stabilizes the point $\bP(\ell)\in \bP(V)$. The set 
$$
\bar{\mathcal{L}}=\bar{\mathcal{L}}(V):=\{\text{pairs }(\Gamma, \epsilon) \text{ as above}\}\cup\{\infty\}
$$
carries information about the subspace arrangement. In particular, we associate to every pair $(\Gamma, \epsilon)\in \bar{\mathcal{L}}$ the subspace 
$$
V_{\Gamma, \epsilon}:=\{v\in V \mid v\cdot g=\epsilon(g)v,\text{ for all }g\in \Gamma\}.
$$
The De Concini--Procesi model $\bP(V)_{\bar{\mathcal L}}$ is defined as the closure of the image of the natural map
$$
\bP(V)^\circ\to\bP(V)\times \prod_{\substack{(\Gamma, \epsilon)\in\bar{\mathcal L}\\\Gamma\ne C}}\bP(V/V_{\Gamma, \epsilon}),
$$
where the $\bP(V)^\circ$ is the complement in $\bP(V)$ of the union of all proper subspaces of the form $\bP(V_{\Gamma,\epsilon})$. The natural projection 
$$
\bP(V)_{\bar{\mathcal L}}\to\bP(V)
$$
is an isomorphism on $\bP(V)^\circ$, whose complement in $\bP(V)_{\bar{\mathcal L}}$ is a normal crossings divisor. It is shown in \cite[Proposition 7.2]{KT-vector} that 
the $G$-action on $\bP(V)_{\bar{\mathcal L}}$ is in standard form with respect to this divisor.
We now describe the main steps of the algorithm.

\

\noindent\textbf{Input.} A faithful linear representation $G\to\GL(V^\vee)$.

\

\noindent\textbf{Step 1.} Find $C$ and $\bar{\mathcal L}=\bar{\mathcal L}(V)$, i.e., all possible pairs $(\Gamma,\epsilon)$ as in \eqref{eqn:gamma}. 

\

\noindent\textbf{Step 2.} Find all chains of subspaces, up to conjugation by $G$, 
$$
0\subsetneq V_1\subsetneq V_2\subsetneq\cdots\subsetneq V_t\subsetneq V
$$
such that
\begin{itemize}
\item $V_i=V_{\Gamma^i,\epsilon}$ for some pair $(\Gamma^i,\epsilon)\in\bar{\mathcal L}$ with $\Gamma^i\ne C$, for every $i=1,\ldots, t$ and a common character $\epsilon$, 
\item $\Gamma^i$ is the (maximal) stabilizer group of $V_i$.
\end{itemize}
Associated with each chain of subspaces is a chain of stabilizer groups, 
$$
\Lambda:=\Gamma^1\supsetneq \Gamma^2\supsetneq\cdots\supsetneq\Gamma^t,
$$ and a character $\epsilon$ of $\Gamma^1$.

\

\noindent\textbf{Step 3.} For each conjugacy class of chains of subspaces $V_1\subsetneq \cdots\subsetneq V_t$ and the corresponding chain of stabilizers $\Lambda$, find 
\begin{itemize}
    \item $N_G(\Lambda)\subseteq G$, the intersection of normalizers of $\Gamma^i$ in $G$ which stabilize $\epsilon$, this is the stabilizer of $\Lambda.$
    \item $\Delta_\Lambda^t$, the maximal subgroup of $N_G(\Lambda)$ acting via scalars on all $V_{i+1}/V_i$.  
\end{itemize}
The input representation induces a faithful representation of $N_G(\Lambda)$ on 
$$
V_1^{\vee}\times (V_2/V_1)^{\vee}\times (V_3/V_2)^{\vee}\times \cdots \times (V/V_t)^{\vee}, 
$$
where $\Delta_\Lambda^t$ acts via scalars on each factor; 
we record characters $\epsilon^i$ of $\Delta_\Lambda^t$ on $V_{i+1}/V_i,
$ $i=0,\ldots,t$.  By convention, $V_0=0$ and $V_{t+1}=V$.\\

\noindent\textbf{Step 4.} For each conjugacy class of chains, compute an {\em{intermediate class}} 
$$
[\bP(V_1)\times\bP(V_2/V_1)\times\ldots\times\bP(V/V_t)\actsfromright N_G(\Lambda)]_{(\cO(-1))}
$$ 
of the induced action of $N_G(\Lambda)$, 
with respect to $(\cO(-1))$, a sequence of line bundles 
$$
\cO_{\bP(V_1)}(-1),\cO_{\bP(V_1)}(1)\otimes\cO_{\bP(V_2/V_1)}(-1), \cO_{\bP(V_2/V_1)}(1)\otimes\cO_{\bP(V_3/V_2)}(-1), \ldots .
$$
This intermediate class takes values in
$$
\Burn_{n,\{0,\ldots,t\}}(N_G(\Lambda), \Delta_\Lambda^t),
$$
the equivariant indexed Burnside group with respect to line bundles $(\cO(-1))$, defined in \cite[Section 4 and Section 5]{KT-vector}.
Since the De Concini--Procesi model satisfies the conditions in \cite[Lemma 5.1]{KT-vector}, we can compute the intermediate class by \cite[Definition 5.3]{KT-vector}. 
\\

\noindent\textbf{Step 5.} 
A recursive formula \cite[Proposition 8.3 and Theorem 8.4]{KT-vector} allows to compute the class
$$
[\bP(V)\actsfromright G]_{(\cO_{\bP(V)}(-1))}\in\Burn_{n,\{0\}}(G,C)
$$
using all intermediate classes of chains found in Step 2. Apply this recursion to obtain this class,
taking values in the equivariant indexed Burnside group with respect to the line bundles $(\cO_{\bP(V)}(-1))$.

\

\noindent\textbf{Step 6.} Apply 
the map
    $$
    \eta_{\{0\}}:\Burn_{n,\{0\}}(G,C)\to\Burn_n({G/C}),
    $$
    defined by
    $$
    (C\subseteq H', Z'\actsfromleft K,\beta, \gamma)\mapsto(H'/C, Z'\actsfromleft K, \beta).
    $$
By \cite[Theorem 8.5]{KT-vector}, we have
$$
[\bP(V)\actsfromright G/C]=\eta_{\{0\}}\left([\bP(V)\actsfromright G]_{(\cO_{\bP(V)}(-1))}\right).
$$

\

\noindent\textbf{Output.} The class 
$$
[\bP(V)\actsfromright G/C]\in\Burn_n(G/C)
$$
is presented as a finite sum of symbols in $\mathrm{Symb}_n(G)$.\\

\

As already noted, an important observation is:

\begin{coro}
\label{coro:yacts}
Every symbol $\mathfrak s$ 
appearing as a summand in the class
$$
[\bP(V)\actsfromright G] \in \Burn_n(G),
$$
via the algorithm from \cite{BnG} is of the shape
$$
\mathfrak s = (H, Y\actsfromleft k(F), \beta),
$$
where
\begin{itemize}
\item 
$F$ is birational to $\prod_j \bP(W_j)$, 
\item $Y\subseteq Z_G(H)/H$ acts without interchanging the factors, and
\item the action on each factor is (birational) to a (projectively) linear action.
\end{itemize}
In particular,  
$$
[\bP(V)\actsfromright G] \in \Burn_n^{\mathrm{rc},\mathrm{H1}=0}(G),
$$
(see Section~\ref{sect:h1-triv}). 
\end{coro}

An example computation, for $G=\fS_4$, acting on $\bP^2=\bP(V)$, where $V$ is the standard 3-dimensional representation of $\fS_4$, can be found in \cite[Section 9]{KT-vector}. 
Here, we provide new examples, in dimensions 2 and 3. 
\begin{exam}
Let $G=C_3\times\fD_5$ acting on $\bP^2=\bP(1\oplus V_{\epsilon})$;
here 
$$
V_{\epsilon}:=\epsilon\otimes V
$$
is the twist by a nontrivial character of $C_3$ of
the standard 2-dimensional representation of $\fD_5$, with generators acting via
$$
\begin{pmatrix}
\zeta_5&0\\
0&\zeta^{-1}_5
\end{pmatrix},
\begin{pmatrix}
0&1\\1&0
\end{pmatrix}.
$$
We tabulate the relevant information for conjugacy classes of chains of stabilizer groups from Steps 1, 2 and 3.
\begin{center}
    \begin{tabular}{c|c|c|c|c}
    $t$ & $\Lambda$ & $N_G(\Lambda)$ & $\Delta_{\Lambda}^t$ & $\epsilon^i$\\
    \hline
     1 &  $C_3\times\fD_{5}$ & $C_3\times\fD_{5}$ & $C_3$ & $0$ \\
     1 &  $C_{15}$ & $C_{15}$ & $\mathrm{trivial}$ & $-$ \\
    1 &  $C_6$ & $C_6$ & $C_2$ & $1$ \\
    1 &  $C_6$ & $C_6$ & $\mathrm{trivial}$ & $-$ \\
    1 & $C_3$ & $C_3\times\fD_5$ & $C_3$ & $1$ \\
    1 & $C_2$ & $C_6$ & $C_2$ & $0$ \\
     2 &  $C_3\times\fD_{5}\supset C_2$ & $C_6$ & $C_6$ & $0,4$ \\
     2 &  $C_{15}\supset C_3$ & $C_{15}$ & $C_{15}$ & $4,1$ \\
    2 &  $C_6\supset C_3$ & $C_6$ & $C_6$ & $1,4$ \\
        2 &  $C_6\supset C_3$ & $C_6$ & $C_6$ & $4,1$ \\
    2 &  $C_6\supset C_2$ & $C_6$ & $C_6$ & $4,0$ \\
    \end{tabular}
\end{center}
Each chain $\Lambda$ contributes to  $[\bP^2\actsfromright G]$ via its intermediate class, obtained in Step 4. 
We record these classes:

\begin{itemize}
    \item $\Lambda=C_3\times\fD_5$:
    \begin{align*}
        &(C_3\subseteq C_3, \fD_5\actsfromleft k(\bP^1), (), (0,2))+(C_3\subseteq C_6, 1\actsfromleft k, (3), (0,2))\\&+(C_3\subseteq C_6, 1\actsfromleft k, (3), (0,5))+(C_3\subseteq C_{15}, 1\actsfromleft k, (9), (0,8))\\
        &\qquad\quad\in\Burn_{3,\{0,1\}}(C_3\times\fD_5,C_3)
    \end{align*}
    \item $\Lambda=C_{15}$:
    \begin{align*}
        &(1\subseteq 1, C_{15}\actsfromleft k(\bP^1), (), (0,0))+(1\subseteq C_{15}, 1\actsfromleft k, (7), (13,2))\\&+(1\subseteq C_{15}, 1\actsfromleft k, (8), (13,9))\in\Burn_{3,\{0,1\}}(C_{15},1)
    \end{align*}
    \item $\Lambda=C_{6}$ with $\Delta_\Lambda^t=1$: \begin{align*}
        &(1\subseteq 1, C_{6}\actsfromleft k(\bP^1), (), (0,0))+(1\subseteq C_{6}, 1\actsfromleft k, (1), (2,3))\\&+(1\subseteq C_{6}, 1\actsfromleft k, (5), (2,4))\in\Burn_{3,\{0,1\}}(C_{6},1)
    \end{align*}
     \item $\Lambda=C_{6}$ with $\Delta_\Lambda^t=C_2$:
     \begin{align*}
        &(C_2\subseteq C_2, C_{3}\actsfromleft k(\bP^1), (), (1,1))+(C_2\subseteq C_{6}, 1\actsfromleft k, (4), (5,3))\\&+(C_2\subseteq C_{6}, 1\actsfromleft k, (2), (5,1))\in\Burn_{3,\{0,1\}}(C_{6},C_2)
    \end{align*}
     \item $\Lambda=C_{3}$:
     \begin{align*}
        &(C_3\subseteq C_3, \fD_5\actsfromleft k(\bP^1), (), (2,1))+(C_3\subseteq C_{6}, 1\actsfromleft k, (3), (2,4))\\&+(C_3\subseteq C_{6}, 1\actsfromleft k, (3), (5,1))+(C_3\subseteq C_{15}, 1\actsfromleft k, (9), (8,7))\\
        &\qquad\quad\in\Burn_{3,\{0,1\}}(C_{3},C_3).
    \end{align*}
\end{itemize}
Our algorithm records the action on function fields in each symbol, e.g., the 
action of $\fD_5$ on $k(\bP^1)$ in the last expression, but we omit it from the notation. 

When $t=2$, each graded piece is a one-dimensional vector space, with $N_G(\Lambda)$ acting via scalars. We will obtain classes
$$(N_G(\Lambda)\subseteq N_G(\Lambda),1\actsfromleft k,(),(\epsilon,\epsilon^1-\epsilon,\epsilon^2-\epsilon^1)).$$ 
Then we use the recursion in Step 5 to compute 
$$
[\bP(V)\actsfromright G]_{(\cO_{\bP(V)}(-1))}\in\Burn_{n,\{0\}}(G,C).
$$
In this example, $G$ acts generically freely on $\bP^2$, so that $C=1$. After applying the map $\eta_{\{0\}}$ in Step 6 and cancellations by relations, we have
\begin{align*}
   [\bP(V)\actsfromright G]&=(1, G\actsfromleft k(\bP^2),())+2( C_{2},C_{3}\actsfromleft k(\bP^1),(1) )\\
   &+(C_3,\mathfrak D_5\actsfromleft k(\bP^1),(2) )+(C_3,\mathfrak D_5\actsfromleft k(\bP^1),(1) )\\&+( C_{6},1\actsfromleft k,(3,2) )+( C_{6},1\actsfromleft k,(3,4) )\\&+( C_{6},1\actsfromleft k,(3,5) )+( C_{6},1\actsfromleft k,(2,1)\\&+( C_{6},1\actsfromleft k,(3,1) )+( C_{6},1\actsfromleft k,(4,5))\\&+(C_{15},1\actsfromleft k,(1,11) )+(C_{15},1\actsfromleft k,(3,11))\\&+(C_{15},1\actsfromleft k,(12,4)).
  \end{align*}
  
There is an alternative method to compute the class  $[\bP(V)\actsfromright G]$ \cite[Section 5]{KT-struct}: 
First, consider the action of $\fD_5$ on $\bP^1$ via its two-dimensional representation $V$.
Let $L_1$ be $\cO_{\bP^1}(1)$ twisted by the nontrivial character $\epsilon$ of $C_3$, and $L_0$ be the trivial line bundle on $\bP^1$. Then $$
\bP(1\oplus V_{\epsilon})\sim_G \bP(L_0\oplus L_1),
$$
equivariantly. Using \cite[Proposition 5.2]{KT-struct}, we obtain 
\begin{align*}
    [\bP(L_0\oplus L_1)\actsfromright G]&=(1, G\actsfromleft k(\bP^2),())+( C_{2},C_{3}\actsfromleft k(\bP^1),(1) )\\&+(C_3,\mathfrak D_5\actsfromleft k(\bP^1),(2) )+(C_3,\mathfrak D_5\actsfromleft k(\bP^1),(1) )\\&+( C_{6},1\actsfromleft k,(3,2) )+( C_{6},1\actsfromleft k,(3,4) )\\&+( C_{6},1\actsfromleft k,(3,5) )+( C_{6},1\actsfromleft k,(3,1) )\\&+(C_{15},1\actsfromleft k,(3,11) )+(C_{15},1\actsfromleft k,(3,4)).
\end{align*} 
Here we specify the subgroups and their representations:
$$
C_3=\left\langle\begin{pmatrix}
\zeta^{2}_3&0\\0&\zeta^{2}_3
\end{pmatrix}\right\rangle,\quad C_{6}=\left\langle\begin{pmatrix}
0&\zeta_3\\\zeta_3&0
\end{pmatrix}\right\rangle\quad C_{15}=\left\langle\begin{pmatrix}
\zeta_3\zeta_5&0\\0&\zeta_3\zeta^4_5
\end{pmatrix}\right\rangle.$$
Note that 
\begin{align*}
    &[\bP(V)\actsfromright G]-[\bP(L_0\oplus L_1)\actsfromright G]\\
    &= (C_2,C_3 \actsfromleft k(\bP^1),(1))+(C_{6},1\actsfromleft k,(2,1) )+(C_{6},1\actsfromleft k,(4,5) )\\
    &+(C_{15},1\actsfromleft k,(1,11))+(C_{15},1\actsfromleft k,(12,4))-(C_{15},1\actsfromleft k,(3,4)).
\end{align*}
By conjugation relations \textbf{(C)},  \begin{align*}
(C_{15},1\actsfromleft k,(3,4) )&=(C_{15},1\actsfromleft k,(12,1) )
\end{align*}
The blow-up relations $\textbf{(B)}$ yield $$(C_{15},1\actsfromleft k,(12,1) )=(C_{15},1\actsfromleft k,(11,1)+(C_{15},1\actsfromleft k,(12,4) ),$$   
$$(C_{6},1\actsfromleft k,(2,3) )=(C_{6},1\actsfromleft k,(5,3) )+(C_{6},1\actsfromleft k,(2,1) ),$$
\begin{multline*}
(C_{6},1\actsfromleft k,(3,5))=  \\
(C_{6},1\actsfromleft k,(3,2) )+(C_{6},1\actsfromleft k,(4,5) )+(C_2,C_3 \actsfromleft k(\bP^1),(1)).
\end{multline*}
Summing up the last two equalities, we obtain 
$$
(C_{6},1\actsfromleft k,(2,1) )+(C_{6},1\actsfromleft k,(4,5) )+(C_2,C_3 \actsfromleft k(\bP^1),(1))=0
$$
and conclude 
$$[\bP(V)\actsfromright G]-[\bP(L_0\oplus L_1)\actsfromright G]=0\in\Burn_2(G),
$$
as expected. 
\end{exam}

\begin{exam}
Consider the action of $G=\fD_7$ on $\bP^3$, given by
\begin{align*}
    G=\left\langle
    \begin{pmatrix}
    \zeta_7 & 0 & 0 & 0\\
    0 & \zeta_7^{-1} & 0 & 0\\
    0 & 0 & \zeta_7^2 & 0\\
    0 & 0 & 0 & \zeta_7^{-2}
    \end{pmatrix},
    \begin{pmatrix}
    0 & 1 & 0 & 0\\
    1 & 0 & 0 & 0\\
    0 & 0 & 0 & 1\\
    0 & 0 & 1 & 0
    \end{pmatrix}
    \right\rangle \subset \PGL_4.
\end{align*}
The stabilizer chains are
\begin{center}
    \begin{tabular}{c|c|c|c|c}
    $t$  & $\Lambda$ & $N_G(\Lambda)$ & $\Delta_{\Lambda}^t$ & $\epsilon$\\
    \hline
    1  & $C_2$ & $C_2$ & $C_2$ & 0 \\

    1  & $C_2$ & $C_2$ & $C_2$ & 1 \\

    1  & $C_7$ & $C_7$ & $C_1$ & 2 \\

    1  & $C_7$ & $C_7$ & $C_1$ & 3 \\
    \end{tabular}
\end{center}

\noindent
The intermediate classes in the equivariant indexed Burnside groups are:

\begin{itemize}
    \item  $\Lambda=C_2 $ with $N_G(\Lambda)=C_2$:
        \begin{align*}
            (C_2 \subseteq C_2,1 \actsfromleft k(\bP^2),(),(0,1)) \in \Burn_{3,\{0,1\}}(C_2,C_2)    
        \end{align*}
    \item  $\Lambda=C_2$ with $N_G(\Lambda)=C_2$:
        \begin{align*}
            (C_2 \subseteq C_2,1 \actsfromleft k(\bP^2),(),(1,1)) \in \Burn_{3,\{0,1\}}(C_2,C_2)
        \end{align*}
    \item $\Lambda=C_7$:
        \begin{align*}
            &(C_1 \subseteq C_1,C_7 \actsfromleft k(\bP^2),(),(0,0))+(C_1 \subset C_7,1 \actsfromleft k,(5,6),(2,3))\\
            &+(C_1 \subseteq C_7,1 \actsfromleft k,(1,2),(2,1))+(C_1 \subseteq C_7,1 \actsfromleft k,(1,6),(2,2)) \\
            &\qquad\in \Burn_{3,\{0,1\}}(C_7,1)
        \end{align*}
    \item $\Lambda=C_7$:
        \begin{align*}
            &(C_1 \subseteq C_1,C_7 \actsfromleft k(\bP^2),(),(0,0))+(C_1 \subseteq C_7,1 \actsfromleft k,(4,6),(3,2))\\
            &+(C_1 \subseteq C_7,1 \actsfromleft k,(2,3),(3,6))+(C_1 \subseteq C_7,1 \actsfromleft k,(1,5),(3,1)) \\
            &\qquad\in \Burn_{3,\{0,1\}}(C_7,1).
        \end{align*}
\end{itemize}
    These classes are combined to obtain
        \begin{align*}
            &[\mathbb{P}(V) \actsfromright G]_{(\mathcal O_{\mathbb{P}(V)}(-1))}=(C_1 \subseteq C_1,G \actsfromleft k(\bP^3),(),(0))\\
            &+(C_1 \subset C_2,1 \actsfromleft k(\bP^2),(1),(0))+(C_1 \subseteq C_2,1 \actsfromleft k(\bP^2),(1),(1))\\
            &+(C_1 \subseteq C_7,1 \actsfromleft k,(3,5,6),(2))+(C_1 \subset C_7,1 \actsfromleft k,(1,1,2),(2))\\
            &+(C_1 \subseteq C_7,1 \actsfromleft k,(1,2,6),(2))+(C_1 \subseteq C_7,1 \actsfromleft k,(2,4,6),(3))\\
            &+(C_1 \subseteq C_7,1 \actsfromleft k,(2,3,6),(3))+(C_1 \subseteq C_7,1 \actsfromleft k,(1,1,5),(3))
        \end{align*}
        Applying $\eta_{\{0\}}$ and using relation $\textbf{(V)}$, we obtain the nonzero class
        \begin{align*}
            [\mathbb{P}(V) \actsfromright G]= & (1,G \actsfromleft k(\bP^3),())\\
            + & (C_7,1 \actsfromleft k,(1,1,2))+(C_7,1 \actsfromleft k,(2,4,6))\\
            + & (C_7,1 \actsfromleft k,(2,3,6)) \in \Burn_3(G);
        \end{align*}
        in fact, the point classes in this formula are equal, and nonzero, in 
        $\BC_3(G)= \bZ/2$. The action is birational to an action on $\bP^1\times \bP^2$, 
        with trivial action on the second factor and faithful action on the first factor, by the No-name Lemma. 
\end{exam}

\section{Automorphisms of $\bP^2$}
\label{sect:dim2}

In this section, we apply the Burnside group formalism to the problem of classification of actions of finite subgroups of $\PGL_3$ up to conjugation in the plane Cremona group $\mathrm{Cr}_2$ (see \cite{DI}).

For $n=2$, the classification of actions up to conjugation in $\PGL_3$ takes the form
(we follow \cite[Section 4.2]{DI} and \cite[Section 10]{KT-vector}):
\begin{itemize}
\item {\em intransitive}: 
$
G=C_m\times G',
$
with $G'\subset \GL_2$,
\item {\em transitive but imprimitive}:
certain extensions of $C_3$ or $\fS_3$ by bi-cyclic groups,
\item {\em primitive}:
$
\mathfrak{A}_5,\, \mathfrak{A}_6, \,  \mathrm{PSL}_2(\mathbb{F}_7),
$
the Hessian group 
$
3^2:\mathrm{SL}_2(\mathbb{F}_3),
$
and two of its subgroups.
\end{itemize}

\noindent
\textbf{Primitive actions.} These are 
completely understood via birational (super)rigidity techniques \cite{sako}.
E.g., $\fA_5$ admits one, $\fA_6$ admits four, and $\mathrm{PSL}_2(\bF_7)$ admits two non-birational actions on $\bP^2$ (see \cite[Theorem B.2]{Chelocal}). 

\begin{prop}
\label{prop:doesnot}
The Burnside group formalism does not distinguish primitive actions on $\bP^2$. 
\end{prop}

The proof proceeds via a computation of all classes involved and comparisons of the resulting expressions in the respective Burnside groups. Here is a representative example:

\begin{exam}
\label{exam:psl27}
The action of $G:=\mathrm{PSL}_2(\mathbb F_7)$ on $\bP^2$ is {\em super-rigid}, and there are non-isomorphic 3-dimensional representations $V$ and $ V'$ of $G$, giving rise to non-birational $G$-actions on $\bP^2=\bP(V)$ and $\bP(V')$. The characters of the corresponding 
representations differ on elements of order 7. 
We compute the classes
        
        \begin{align*}
            [\mathbb{P}(V) & \actsfromright G]=(1,G \actsfromleft k(\bP^2),())+2(C_2,\fD_2 \actsfromleft k(\bP^1),(1))\\
            &+(C_3,1 \actsfromleft k,(1,1))+(C_4,1 \actsfromleft k,(1,1))+2(C_4,1 \actsfromleft k,(1,2))\\
            &+(C_7,1 \actsfromleft k,(6,5))+(C_7,1 \actsfromleft k,(1,4))\\
            &+(C_2^2,1 \actsfromleft k,((0,1),(1,0)))+((C_2')^2,1 \actsfromleft k,((0,1),(1,0)))     
        \end{align*}
        \begin{align*}
            [\mathbb{P}(V') & \actsfromright G]=(1,G \actsfromleft k(\bP^2),())+2(C_2,\fD_2 \actsfromleft k(\bP^1),(1))\\
            &+(C_3,1 \actsfromleft k,(1,1))+(C_4,1 \actsfromleft k,(1,1))+2(C_4,1 \actsfromleft k,(2,3))\\
            &+(C_7,1 \actsfromleft k,(6,3))+(C_7,1 \actsfromleft k,(1,2))\\
            &+(C_2^2,1 \actsfromleft k,((0,1),(1,0)))+((C_2')^2,1 \actsfromleft k,((1,1),(1,0))).     
        \end{align*}
    The representations $V$ and $V'$ differ by $\zeta_7 \mapsto \zeta_7^3$. Conjugation relations imply that 
    \begin{align*}
        [\mathbb{P}(V) \actsfromright G]=[\mathbb{P}(V') \actsfromright G].
    \end{align*}
\end{exam}

We record useful method to produce incompressible classes in dimension 3 (see Section~\ref{sect:inc}). 

\begin{prop}
\label{prop:cr}
Let $G$ be a finite group and 
$$
\bar{s} = (\bar{H},\bar{Y} \actsfromleft k(\bP^1)(t), (\bar{b})) \in 
\Burn_3(G) 
$$
a symbol appearing in a $\Theta_2$-relation.   
Then $\bar{Y}$ does not admit a primitive action on $\bP^2$. 
\end{prop}

\begin{proof} 
By classification, we know that an $\fA_5$-action on $\bP^2$ is not birational to an action on $\bP^1\times \bP^1$ \cite[Theorem~6.6.1]{CS};
$\fA_6$ and $\mathrm{PSL}_2(\mathbf F_7)$ cannot act on $\bP^1$ and thus not on the projectivization of a sum of line bundles over $\bP^1$. 
A similar argument applies to subgroups of the Hessian group (which admit a primitive action on $\bP^2$). 
\end{proof}

\noindent
\textbf{Transitive Imprimitive actions.} 
There are four types of such actions, two types with $G$ an extension of $C_3$ and
two additional types when $G$ is an extension of $\fS_3$, see \cite[Theorem 4.7]{DI}.

\begin{prop}
The Burnside group formalism allows to distinguish 
transitive imprimitive actions, indistinguishable by 
the $\mathbf{(RY)}$ invariant.
\end{prop}

We do not claim that we can distinguish {\em all} such actions. 
In each of the four types there is a bi-cyclic group $H\subset G$; restricting to $H$ and applying the Reichstein-Youssin determinant invariant $\textbf{(RY)}$ to $H$ 
gives non-birational actions in some cases.
Our examples focus on the simpler types in \cite[Theorem 4.7]{DI},
as it is more difficult to distinguish smaller actions. 

\

We consider: 

\begin{itemize}
        \item [(1)] extensions 
        $$
        1\to C_n\oplus C_n\to G\to C_3 \to 1
        $$
        with the action on $\bP^2=\bP^2(s,t)$ given by 
        \begin{align}
        \label{eqn:act1}
         (x:y:z)\mapsto   (\zeta_n^s x:y:z),\ (x:\zeta_n^t y:z),\ (z:x:y),
        \end{align}
        where  $s,t \in (\bZ/n)^\times$, and $\zeta_n$ is a primitive $n$-th root of unity. 
       \item [(2)] extensions
         $$
        1\to C_n\oplus C_m\to G\to C_3 \to 1,
        $$
       with $m=n/d$, with $d>1$, $d|n$, $s^2-s+1=0\ (\mathrm{mod}\ d)$, 
       and with the action on $\bP^2=\bP^2(r,s,t)$ via
        \begin{align}\label{eqn:act2}
         (x:y:z)\mapsto     (\zeta_m^rx:y:z),\ (\zeta_n^sx:\zeta_n^ty:z),\ (z:x:y).
       \end{align}

   \end{itemize}

\begin{exam}
Let $G$ be a group of type $(1)$, with $n=8$. Consider actions as in \eqref{eqn:act1} with 
and 
$$
s=1,\quad t=7,
$$
respectively, 
$$s'=3, \quad t'=5.
$$
The $\textbf{(RY)}$ 
invariant is inconclusive in this case. Computing the Burnside symbols as in Section \ref{sect:com-lin}, we obtain
 \begin{align*}
   [\bP^2(s,t)&\actsfromright G]
   = (1, G\actsfromleft k(\bP^2), ()) 
\\&+    
   ( C_{8},C_8\actsfromleft k(\bP^1),(3) )+( C_{8},C_8\actsfromleft k(\bP^1),(5) )
  \\&+( C_8^2,1\actsfromleft k,((1,2),(6,7)) )+( C_8^2,1\actsfromleft k,((7,6),(7,1)) ).
  \end{align*}
  \begin{align*}
   [\bP^2(s',t')&\actsfromright G]
    = (1, G\actsfromleft k(\bP^2), ()) 
\\&+ 
  ( C_{8},C_{8}\actsfromleft k(\bP^1),(1) )+( C_{8},C_{8}\actsfromleft k(\bP^1),(7) )
  \\&+( C_8^2, 1\actsfromleft k,((3,6),(2,5)) )+( C_8^2,1\actsfromleft k,((5,2),(5,3)) ).
  \end{align*}
  (As before, we omit to specify the action of $C_8$ on $k(\bP^1)$ from our notation.) 
  There are no incompressible symbols in the expressions above, however we are still able to distinguish the actions in the combinatorial Burnside group, after 
  applying map \eqref{eq:BurntoBC} to the difference 
  $$
  [\bP^2(s,t)\actsfromright G]-[\bP^2(s',t')\actsfromright G],
  $$
 and performing {\tt{magma}} computations in $\BC_2(G)$. 
 
The same argument applies to $n=5$, $s=1, t=2, s'=3,$ and $t'=4$; or $n=9$, $s=2, t=3, s'=4, $ and $t'=6$. 
  \end{exam}

\begin{exam}
Let $G$ act via type (2) with $n=14$ and $m=2$. Consider actions as in 
\eqref{eqn:act2} with 
$$
r=t=1, s=3,
$$
respectively, 
$$
r'=t'=1, s'=5.
$$
Again, the $\textbf{(RY)}$ 
invariant is inconclusive.
We have
 \begin{align*}
   [\bP^2(r,s,t)\actsfromright G]
   &=(1,G\actsfromleft k(\bP^2),())\\
   &+( C_{2},C_{14}\actsfromleft k(\bP^1),(1) )+( C_{2},C_{14}\actsfromleft k(\bP^1),(1) )
  \\&+( C_2\times C_{14},1\actsfromleft k,((0,3),(1,5)) )\\&+( C_2\times C_{14},1\actsfromleft k,((0,11),(1,8)) ),
  \end{align*}
  \begin{align*}
   [\bP^2(r',s',t')\actsfromright  G]
   &=(1,G\actsfromleft k(\bP^2),())\\
   &+( C_{2},C_{14}\actsfromleft k(\bP^1),(1) )+( C_{2},C_{14}\actsfromleft k(\bP^1),(1) )
  \\&+( C_2\times C_{14},1\actsfromleft k,((1,11),(0,1)) )\\&+( C_2\times C_{14},1\actsfromleft k,((1,3),(1,12)) ).
  \end{align*}
Applying map \eqref{eq:BurntoBC} to the difference and computing in $\BC_2(G)$ we find 
that the actions are non-birational.
  
\end{exam} 

\

 \noindent \textbf{Intransitive actions.} Existence of $G$-fixed points makes it more difficult to classify intransitive actions using birational rigidity techniques. However, it is well-suited for the Burnside group formalism. Recall that intransitive actions take the form of 
 $$
G=C_n\times G',\quad n\geq 2,
 $$
where $G'\subset\GL_2$ is a lift of a subgroup $\bar G'\subset \mathrm{PGL_2}$. We are again in the situation of Section \ref{sect:dim1}:
\begin{itemize}
    \item $\bar G'=C_m$ for some $m\geq 2$. Then $G'$ is also a cyclic group, i.e., $G$ is a rank $2$ abelian group. The $\textbf{(RY)}$ invariant determines equivariant birationality of such actions  \cite[Theorem 7.1] {reichsteinyoussininvariant}.
    \item $\bar G'=\fD_m,\fA_4,\fS_4$ or $\fA_5$. By \cite[Section 10]{KT-vector}, we know that $G$ admits non-birational actions when $\varphi(n)\geq 3$.
    Here we modify the proof to cover more cases when $n\ge 2$.
\end{itemize}
    Let $\epsilon$ be a primitive character of $C_m,$ $V$ a faithful two-dimensional linear representation of $G'$, and $V_\epsilon:= \epsilon \otimes V$ its twist by $\epsilon$. This yields 
    generically free action $G$-action on $\bP^2=\bP(1\oplus V_\epsilon)$. To put the action in standard form, we first need to blow up the point $(1:0:0)$ as it has nonabelian generic stabilizer. The action on the exceptional divisor is given by $\bP(V_\epsilon)$. On the standard model, 
    there are two divisors with generic stabilizer $H$, where $H$ is the maximal subgroup of $G$ acting via scalars on $V_\epsilon$. For example, when $\bar G'=\fA_5$, we can choose the lift $G'=\SL_2(\mathbb F_5)$ and in this case,  $$
    H=\begin{cases}
C_n&\text{ when } n \text{ is even, }\\
C_{2n}&\text{ when } n \text{ is odd.}
\end{cases}$$
Let $\chi_\epsilon$ be the character of $H$ corresponding to the action, which depends on choice of $\epsilon$. The two divisors contribute  
\begin{align}
\label{eq:Cninc}
   (H,\bar G'\actsfromleft k(\bP(V)),(\chi_\epsilon))+(H,\bar G'\actsfromleft k(\bP(V)),(-\chi_\epsilon))
\end{align}
to the class $[\bP^2\actsfromright G]$;   
these symbols are incompressible, as explained in Section \ref{sect:inc}. When $\varphi(n)\geq 3$, we can produce non-birational actions by choosing characters $\epsilon\neq \pm \epsilon'$. But one can do better: 



\begin{coro}\label{coro:incbir}
For $\bar G'=\fD_m$, with $m\ne 1,2,3,4,6,8,12$, or $\bar G'=\fA_5$, and all $n\ge 2$, the group 
$G=C_n\times G'$ admits non-birational linear actions on $\bP^2$.
\end{coro}

\begin{proof}
From Section \ref{sect:dim1}, we know that $\fD_m$, with $m$ as in the statement, and $\fA_5$ admit non-birational actions on $\bP^1$. 
This will contribute different incompressible symbols to \eqref{eq:Cninc}.
\end{proof}

Now we consider the case $\bar G'=\fD_m$ in more detail. 
Recall that for $m$ odd, a generically free action of $\fD_m$ on $\bP^1$ is linear; for $m$ even, it is projectively linear---it arises from a 2-dimensional faithful representation of $\fD_{2m}$. 
In both cases, the representation is determined by a primitive character
$\psi$ of $C_m$, respectively $C_{2m}$, we denote it by $V_{\psi}$.
We obtain an action of  $G=C_n\times \fD_m$ on 
$$
\bP^2=\bP^2(\epsilon,\psi):=\bP(1\oplus V_{\epsilon,\psi}), \quad V_{\epsilon,\psi}:= \epsilon \otimes V_{\psi}.
$$

\begin{lemm}
\label{lemm:Dmpm}
We have
$$
\bP^2(\epsilon,\psi) \sim_G 
\bP^2(-\epsilon,\psi) \sim_G
\bP^2(\epsilon,-\psi) \sim_G
\bP^2(-\epsilon,-\psi).
$$
\end{lemm}
\begin{proof}
Indeed, equivariant birationality from the $G$-action on  $\bP^2(\epsilon,\psi)$
to the other actions is realized by 
$$
(x:y:z)\dashrightarrow(\frac1x:\frac1z:\frac1y),\,(x:z:y),\text{ and }(\frac1x:\frac1y:\frac1z),
$$ 
respectively. 
\end{proof}

\

\noindent
\textit{$m$ is odd:}
The following sum of 
incompressible symbols \begin{align}
\label{eq:incompDmodd}
 (C_n,\fD_m\actsfromleft k(\bP(V_\psi)),\epsilon)+(C_n,\fD_m\actsfromleft k(\bP(V_\psi)),-\epsilon)
 \end{align}
 contributes to the class
 $[\bP^2(\epsilon,\psi)\actsfromright G]$;
 we obtain similar expressions for the $G$-action on $\bP^2(\epsilon',\psi')$. 
We observe:
 \begin{itemize}
 \item when $\epsilon\neq \pm \epsilon'$, 
 the symbols in \eqref{eq:incompDmodd} have different weights; 
 \item 
 when $\psi\ne\pm\psi'$, the actions of $\fD_m$ on $\bP^1$ is not birational to each other.
 \end{itemize}
 Lemma \ref{lemm:Dmpm} implies that the Burnside group formalism determines equivariant birationality in this case. 
 
 On the other hand, when $m$ is {\em even}, the classification of equivariant birational types remains open: 
 
\begin{exam}
\label{exam:inter}
Consider $G=C_3\times\fD_8$, and put $\psi':=\psi^3$. Then 
$$
[\bP^2(\epsilon, \psi) \actsfromright G]-[\bP^2(\epsilon, \psi') \actsfromright G] = 0\in \Burn_2(G).
$$
However, we cannot tell whether or not 
$$
\bP^2(\epsilon, \psi) \stackrel{\!\!\!?}{\sim_G} \bP^2(\epsilon, \psi').
$$
In detail, 
\begin{align*}
   &[\bP^2 (\epsilon, \psi)\actsfromright G]=(1, G\actsfromleft k(\bP^2),())\\ & +2(C_{2}, C_{6}\actsfromleft k(\bP^1),(1) )+2(C_{2}',C_{6}\actsfromleft k(\bP^1),(1) )\\
   &+(C_6,\mathfrak D_4\actsfromleft k(\bP^1),(1) )+(C_6,\mathfrak D_4\actsfromleft k(\bP^1),(5) )\\
   &+(C_2''\times C_{6},1\actsfromleft k,((0,3),(1,5)) )+( C_2''\times C_{6},1\actsfromleft k,((1,2),(1,1)) )\\
   &+( C_2''\times C_{6},1\actsfromleft k,((1,4),(0,3)) )+( C_2'''\times C_{6},1\actsfromleft k,((1,2),(0,3)) )\\
   &+(C_2'''\times C_{6},1\actsfromleft k,((1,5),(1,4)) )+( C_2'''\times C_{6},1\actsfromleft k,((1,1),(0,3)) )\\&+(C_{24},1\actsfromleft k,(19,11))+(C_{24},1\actsfromleft k,(5,6))+(C_{24},1\actsfromleft k,(19,18)),
  \end{align*}
 while 
  \begin{align*}
   &[\bP^2(\epsilon, \psi')\actsfromright G]\\
   &=(1, G\actsfromleft k(\bP^2),()) + \ldots \\  
 &+(C_{24},1\actsfromleft k,(6,17))+(C_{24},1\actsfromleft k,(7,23))
+(C_{24},1\actsfromleft k,(7,18)),
  \end{align*}
with the only difference in the sum of terms with stabilizer $C_{24}$, and these expressions are equal in $\Burn_2(G)$. 

\end{exam}

\section{Automorphisms of $\bP^3$}
\label{sect:3}

In this section, we give new examples of non-birational imprimitive linear actions on $\bP^3$. The basic terminology is as follows:
$$
\text{actions}
\begin{cases}
\text{intransitive: invariant point or line} \\
\text{transitive:} 
\quad
\begin{cases}
\text{imprimitive:}\quad 
\begin{cases}
\text{2 skew lines} \\
\text{orbit of length 4 (monomial)}
\end{cases}
\\
\text{primitive: none of the above} 

\end{cases}
\end{cases}
$$

\

\noindent\textbf{Primitive actions.} We follow \cite{CS-finite}. 
There are 30 conjugacy classes of finite subgroups $G\subset \PGL_4$ yielding primitive actions. They are listed, with inclusions, in \cite[Appendix A]{CS-finite}.  
These actions can be analyzed by birational (super)rigidity techniques, see \cite{CheShr} or \cite{CS-finite}. By \cite[Theorem 1.1]{CS-finite}, the action is birationally {\it rigid} iff $G\neq \fA_5$ or $\fS_5$. This means that
applying $G$-MMP to any $G$-birational model one is reduced to $\bP^3$; but this does not imply that different actions on $\bP^3$ are equivariantly birational.  
We now list representative computations of Burnside classes: 

\begin{itemize}
\item 
$G:=\fA_5$: Let $V$ be its irreducible 4-dimensional representation. Consider the induced action on $\bP^3=\bP(V)$. Then  
\begin{align*}
        & [\mathbb{P}^3 \actsfromright G]=
        (1,\fA_5 \actsfromleft k(\bP^3),()) 
        +2(C_2,C_2 \actsfromleft k(\bP^2),(1)) \\
        &+(C_3,1 \actsfromleft k(\bP^1),(2,2))\!+\!(C_3,1 \actsfromleft k(\bP^1),(1,1))\\
        &+(C_5, \actsfromleft k,(1,1,1))+(C_5,C_5 \actsfromleft k,(2,2,4))
    \end{align*}
By Lemma~\ref{lemm:triv-point}, the point classes are trivial; furthermore,
$$
(C_2,C_2 \actsfromleft k(\bP^2),(1)) = 
(C_2,C_2 \actsfromleft k(\bP^1),(1,1)) = 0 \in \Burn_3(G), 
$$
$$
(C_3,1 \actsfromleft k(\bP^1),(b,b)) = 
(C_3,1 \actsfromleft k,(b,b,b))=0 \in \Burn_3(G),
$$
by $\mathbf{(B)}$ and the vanishing relation $\mathbf{(V)}$.
\item $G=\mathrm{PSL}_2(\mathbb F_7)$: The $G$-action on $\bP^3$ is super-rigid \cite[Theorem 1.3]{CS-finite}, but 
every faithful action gives
$$
[\bP^3\actsfromright G]=(1, G\actsfromleft \bP^3,()) \in \Burn_3(G).
$$
\item 
$G=\fA_6$: There are only two actions; they are rigid but not super-rigid, and thus equivariantly birational.
The corresponding classes are 
    \begin{align*}
        [\mathbb{P}^3 \actsfromright G]&= 
        (1, G\actsfromleft k(\bP^3),())
        +(C_3,C_3 \actsfromleft k(\bP^2),(2)),\\
        &+(C_3^2,1 \actsfromleft k,((2,2),(0,1),(2,1)))\\
        &+(C_3^2,1 \actsfromleft k,((0,2),(2,1),(2,2))) . 
    \end{align*}
      \begin{align*}
        [\mathbb{P}^3 \actsfromright G]&= 
        (1, G\actsfromleft k(\bP^3),())
        +(C_3',C_3 \actsfromleft k(\bP^2),(2)),\\
        &+(C_3^2,1 \actsfromleft k,((0,2),(1,1),(1,0))) \\
        &+(C_3^2,1 \actsfromleft k,((0,2),(1,0),(2,2))),
    \end{align*}
    and the nontrivial contributions to their classes in $\BC_3(G)$ are equal, as expected.  But they are nontrivial in this group.
\item $G=\fS_6$: There are two actions, with 
Burnside classes
\begin{align*}
    [\mathbb{P}^3 \actsfromright G]&=(C_1,\fS_6 \actsfromleft k(\mathbb{P}^3),())\\
    &+(C_2,\fA_4 \actsfromleft k(\bP^2),(1))+(C_2',\fA_4 \actsfromleft k(\bP^2),(1))\\
    &+(C_2'',C_2^2 \actsfromleft k(\bP^2),(1))+(C_3,\fS_3 \actsfromleft k(\bP^2),(1))\\
    &+(C_3^2,1 \actsfromleft k,((1,1),(1,2),(2,0))),
\end{align*}
respectively, 
\begin{align*}
    [\mathbb{P}^3 \actsfromright G]&=(C_1,\fS_6 \actsfromleft k(\mathbb{P}^3),())\\
    &+(C_2,\fA_4 \actsfromleft k(\bP^2),(1))+(C_2',\fA_4 \actsfromleft k(\bP^2),(1))\\
    &+(C_2'',C_2^2 \actsfromleft k(\bP^2),(1))+(C_3',\fS_3 \actsfromleft k(\bP^2),(2))\\
    &+(C_3^2,1 \actsfromleft k,((0,2),(2,0),(2,2)))
\end{align*}
These differ in $\BC_3(G)=(\bZ/2)^5\oplus \bZ/4$; thus, the actions are not birational. 
\item $G=\fA_7$: There are two actions. The actions are super-rigid and thus not birational to each other. 
The respective classes are:
    \begin{align*}
        [\mathbb{P}^3 \actsfromright G]&=(1, G\actsfromleft k(\bP^3),()) +(C_2,\fS_3 \actsfromleft k(\bP^2),(1))\\
       &+(C_3,\fA_4 \actsfromleft k(\bP^2),(2))\\
       &+(C_7,1 \actsfromleft k,(2,4,4))+(C_7,1 \actsfromleft k,(1,3,5))\\
       &+(C_7,1 \actsfromleft k,(2,3,3))\\
       &+(C_3^2,1 \actsfromleft k,((0,1),(1,1),(2,0)))\\
       &+(C_3^2,1 \actsfromleft k,((0,1),(2,0),(2,2))),
    \end{align*}
    \begin{align*}
        [\mathbb{P}^3 \actsfromright G]&=(1, G\actsfromleft k(\bP^3),())+(C_2,\fS_3 \actsfromleft k(\bP^2),(1))\\
        &+(C_3,\fA_4 \actsfromleft k(\bP^2),(2))\\
        &+(C_7,1 \actsfromleft k,(2,4,4))+(C_7,1\actsfromleft k,(1,3,5))\\
        &+(C_7,1 \actsfromleft k,(2,3,3))\\
        &+(C_3^2,1 \actsfromleft k,((0,1),(1,0),(2,1)))\\
        &+(C_3^2,1 \actsfromleft k,((0,1),(1,0),(1,2))).
    \end{align*}
We have $\BC_3(G)=(\bZ/2)^3\oplus \bZ$, the (nontrivial contributions to) combinatorial Burnside classes of the two actions are equal, which in this case implies that the classes are equal in $\Burn_3(G)$.  
\end{itemize}

\noindent\textbf{Transitive imprimitive actions.}
Recall that these are of two types:
\begin{itemize}
\item leaving invariant a union of two skew lines,
\item having an orbit of length 4 (monomial subgroups)
\end{itemize}
The second type was analyzed in \cite{CSar}; by its main theorem, every imprimitive monomial subgroup, with the exception of ({\tt GAP ID})
$$
G_{48,3}, \quad G_{9 6,72}, \quad \text{ or } \quad G_{324,160},
$$
is {\it $G$-solid} (i.e., not $G$-birational to conic bundles or Del Pezzo fibrations). Examples of non-birational actions 
are given in \cite[Example 1.6, 1.7 and 1.8]{CSar}. 

Here we present applications of the Burnside group formalism to actions leaving invariant two skew lines.

\begin{exam}
\label{exam:lines}
Let $G:=\fD_5\times\fD_4$ and write $\psi_m$ for a primitive characters of $C_m$. As in Section~\ref{sect:dim2}, let $V_{\psi}$ be a faithful 2-dimensional representation of $\fD_m$ determined by $\psi_m$. 

We have generically free linear $G$-actions on 
\begin{equation}
\label{eqn:act3}
\bP^3=\bP(V_{\psi_5} \oplus V_{\psi_4}), \quad 
\text{respectively,} \quad 
\bP^3=\bP(V_{\psi_5^2} \oplus V_{\psi_4}).
\end{equation}
Our algorithm presents the class of each action in \eqref{eqn:act3} as a sum of more than $60$ symbols;  
we have listed them at \cite{TYZ-web}. 
Again, with {\tt{magma}}, we find that the projection of the difference 
of the classes to $\BC_3(G)$ is nonzero and we conclude that the actions are not birational.

This is the smallest such example we could find; 
the same holds for $G:=\fD_7\times\fD_4$ (and 
$\psi_5$ replaced by $\psi_7$). 
\end{exam}

\noindent\textbf{Intransitive actions:} The discussion is similar to that in Section \ref{sect:dim2}. In dimension $3$, intransitive actions take the form of 
$$
G=C_n\times G',\quad n\geq 2,
$$ 
where $G'\subset \GL_3$ is a lift of $\bar G'\subset \mathrm{PGL}_3$. It is shown in \cite[Theorem 11.2]{KT-vector} that when 
$$
\bar G'=\fS_4,\quad \fA_5, \quad \mathrm{PSL}_2(\bF_7),  \quad \fA_6 \quad\text{ and }\quad \varphi(n)\geq 3,
$$
$G$ admits non-birational actions. Here we use the same argument to cover more cases again: Let $V$ be a $3$-dimensional faithful representation of $G'$ and $\epsilon$ a primitive character of $C_n$. Let $V_\epsilon:=\epsilon \otimes V$ and consider the action $\bP(1\oplus V_\epsilon)$. We need to blow up the fixed point $(1:0:0:0)$ to put the action into standard form and on the blow-up model, there will be two divisors with generic stabilizer $H$, where $H$ is the maximal subgroup of $G$ acting via scalars. Their contribution to the class is 
$$
(H,\bar G'\actsfromleft k(\bP(V)),(\chi_\epsilon))+(H,\bar G'\actsfromleft k(\bP(V)),(-\chi_\epsilon)).
$$ 
These symbols are incompressible for our choice of $\bar G'$ because $\mathrm{PSL_2}(\bF_7)$ and $\fA_6$ are nonabelian and cannot act generically freely on $\bP^1$ (see Proposition~\ref{prop:cr}).  Actions of $\fS_4$ and $\fA_5$ on $\bP^1\times\bP^1$ with trivial action on one factor and generically free actions on the other factor are not linearizable. Similarly to Corollary \ref{coro:incbir}, we know that if $\bar G'$ admits non-birational actions on $\bP^2$, then $G$ admits non-birational actions on $\bP^3$. Keeping the notation above, we arrive at:

\begin{coro}
For $G=C_n\times G'$, with $\bar G'=\mathrm{PSL}_2(\bF_7)$ or $\fA_6$, 
there exist non-birational 
intransitive $G$-actions on $\bP^3$, for all $n\ge 2$.  \end{coro}

\begin{proof}
As in Section \ref{sect:dim2}, these choices of $\bar G'$ give non-birational actions on $\bP^2.$
\end{proof}

\section{Automorphisms of quadrics}
\label{sect:quad}

There is an extensive literature on birationality of quadrics over nonclosed fields (see, e.g., \cite{totaro}); of course, this is only interesting in absence of $k$-rational points. One of the central problems there is the following.

\

\noindent
{\bf Zariski problem for quadrics:}
If two smooth quadrics of the same dimension,
over a nonclosed field, are stably birational then they are birational. 

\

This is known in dimensions $\le 7$. 
On the other hand, in the $G$-equivariant context, there are examples of stably equivariantly birational but not birational quadrics, already in dimension 2. 
Their equivariant geometry has been addressed in, 
e.g., \cite{isk-s3}, \cite{sari},  \cite[Section 7]{HT-torsor}.  
In particular, the quadric surface $Q=\bP^1\times \bP^1$ admits actions of $G=C_2\times \mathfrak D_n$, for odd $n$, which are not birational to linear actions but such that the $G$-action 
on $Q\times \bP^2$, with trivial $G$-action on the second factor, is birational to a linear action
\cite{lemire}, \cite{HT-torsor}. The existence of such stable birationalities makes the analysis of $\Burn_n^{\mathrm{rc}}(G)$, $n\ge 3$, challenging, as one has to account for all such possibilities.

We are not aware of a systematic study of  $G$-equivariant geometry of quadrics in higher dimensions. In particular, it would be interesting to study systematically constructions of $G$-equivariant (stable) birationalities to projective spaces which do not rely on existence of $G$-fixed points.

\

\noindent
{\bf Assumptions on fixed points:}
Projection from fixed points gives trivially linearizability of the action, thus we will assume that 
\begin{itemize}
\item $X^G=\emptyset$.
\end{itemize}
On the other hand, existence of fixed points is a birational invariant for actions of abelian groups, and linear actions of abelian groups have fixed points, thus we will assume that
\begin{itemize}
\item $X^H\neq \emptyset$, for all abelian $H\subset G$. 
\end{itemize}

In this section we consider the birational classification of automorphisms of quadrics from the perspective of Burnside groups. In particular, we focus on $G$-actions satisfying the assumptions above. 

\

\noindent
{\bf Conics:} 
Consider $X \subset \mathbb{P}^2$, given by
\begin{align*}
    \sum_{j=1}^3x_j^2=0, 
\end{align*}
with an action of a subgroup $G$ of the Weyl group $W(\mathsf D_3)=\fS_4$.
The group $W(\mathsf D_3)$ has 11 conjugacy classes of subgroups.
Only one satisfies the requirements (concerning fixed points), namely $\fS_3=\langle \sigma,\tau\rangle$, with $\tau^2=\sigma^3=1$, and the natural permutation action on the coordinates; this action is linearizable. 
We turn to quadric surfaces. 

\

\noindent
{\bf Abelian actions on $\bP^1\times \bP^1$:} 
Their birational classification is in \cite[Proposition 6.2.4] {blancthesis}. In \cite[Section 5.5]{HKTsmall} we noted that the following actions of $C_2^2$ on $\bP^1\times \bP^1$ are not distinguishable with the Burnside formalism: the product action has fixed points, while the diagonal action does not, thus the actions are not birational, but the projections of the classes to the nontrivial part of the Burnside group vanish. 

On the other hand, consider the following, nonlinearizable, actions of $C_2^3$ on $\bP^1\times \bP^1$: in the first case, via
$\mathfrak K_4=C_2^2$ on one factor and $C_2$ on the other factor, and in second case via
$\mathfrak K_4$ on both factors, together with a switch of the factors. 
In the first case, we record
$$
2(C_2, \mathfrak K_4\actsfromleft k(\bP^1), (1)),
$$
coming from the two fixed points on the second $\bP^1$, and  in 
the second case only {\em one} such class. 
Since this symbol is incompressible (see \cite[Proposition 3.6]{KT-vector}), we conclude that the two actions have different classes in the Burnside group.

\

\noindent
{\bf Nonabelian actions on $\bP^1\times \bP^1$:} 
A full list of such actions is given in \cite[Theorem 4.9]{DI}. Here we 
consider the quadric surface $Q$ given by
\begin{equation}
\label{eqn:q2}
\sum_{j=1}^4 x_j^2=0.    
\end{equation}
We focus on actions changing signs and permuting the variables. 
There are 2 conjugacy classes of such groups $G$ satisfying the assumptions on fixed points, namely:

\[
\begin{tikzpicture}[commutative diagrams/every diagram]
	\node (P1) at (2,0) {$\frak{S}_3$,};
	\node (P6) at (0,0)
	{$\mathfrak{D}_6$};
	\path[commutative diagrams/.cd, every arrow, every label]
	(P6) edge node {} (P1);
\end{tikzpicture}
\]

\noindent
where 
$$
\mathfrak D_6=C_2\times \fS_3 =\langle \iota, \sigma, \tau\rangle, \quad \tau^2=\sigma^3=1. 
$$ 
Here $\iota$ inverts the sign on $x_4$, $\fS_3=\langle \sigma,\tau\rangle$ acts via permutation of the first three coordinates, and 
the specialization is to $\fS_3=\langle \sigma, \iota\cdot\tau 
\rangle$.

The fixed-point free $\fS_3$-action is linearizable; it is birational to an action on $\bP(1\oplus V_2)$, where $V_2$ is the standard 2-dimensional representation of $\fS_3$; in particular, there is a fixed point on $\bP^2$.

On the other hand,  by \cite[Section 9]{lemire} (see also \cite[Section 6]{HT-torsor}), 
the $\mathfrak{D}_6$-action on $Q$ is not linearizable but stably linearizable. The proof of nonlinearizability in \cite{isk-s3} was based on classification of birational transformations (links) between rational surfaces. 
An alternative proof, using the 
Burnside group formalism, is in 
\cite[Section 7.6]{HKTsmall}; we 
give a similar argument in the following example.

\begin{exam}
\label{exam:qqq}
Let $G=C_2^2\times \fS_3$. We analyze whether or not the symbol
$$
\bar{\mathfrak s}=(C_2, C_2\times \fS_3 \actsfromleft \bar{K}, (1)) \in \Burn_3(G), \quad \bar{K}=k(Q),
$$
is incompressible. 
There is a candidate symbol
$$
\mathfrak s=(C_2^2, \fS_3\actsfromleft K, (e_1,e_2)),
$$
that could lead to the given $\bar{\mathfrak s}$ via the blowup relation 
$\textbf{(B)}$. Here $e_1,e_2$ are nontrivial distinct characters of $C_2^2$.

Let us specify the action of 
$\bar{Y}=C_2\times \fS_3$ on $\bar{K}=k(Q)$, with $Q$ the quadric surface in \eqref{eqn:q2}:
$C_2$ switches the sign on $x_4$ and $\fS_3$ permutes the first three coordinates.

Since $Q$ is rational, we must have 
$K=k(\bP^1)$. The Action construction produces $\Theta_2$-terms where the $\bar{Y}$-action is birational to an action on a Hirzebruch surface $F$, 
a projectivization of a rank-2 vector bundle on $\bP^1$, either with trivial action or a $C_2$-action on the generic fiber. 

In the first case, such an action is birational to an action on $\bP^1\times \bP^1$, with 
$C_2\times \fS_3 = \fD_6$ acting on one of the factors, and trivial action on the second factor. This action 
has no fixed points upon restriction to $C_2\times \fS_2\subset C_2\times \fS_3$, which is not the case for the $\bar{Y}$ action on $Q$. Thus the actions are not birational.

In the second case, we compare the classes in $\Burn_2(\bar{Y})$, for the actions on $Q$ and on $F$. We find {\em one} incompressible symbol 
$$
(C_2, \fS_3\actsfromleft k(\bP^1), (1)) \in \Burn_2(C_2\times \fS_3)
$$
in the class $[Q\actsfromright \bar{Y}]$, 
and {\em two} such symbols, corresponding to the two sections fixed by $C_2$, 
in the class $[F\actsfromright \bar{Y}]$ (see Section~\ref{sect:inc}). Thus the actions are not birational and $\bar{\mathfrak s}$ is incompressible.
\end{exam}

\

\noindent
{\bf Quadric threefolds:} We consider first $X\subset \bP^4$ given by
 $   \sum_{j=1}^{5}x_j^2=0$,
with a natural action of the Weyl group $W(\mathsf D_5)$. This group has 197 conjugacy  classes of subgroups, examined in \cite[Section 5]{Kun} in connection with the analysis of possible Galois actions (or automorphisms) on Picard groups of Del Pezzo surfaces of degree 4; the goal there was to identify potentially rational surfaces  over nonclosed fields (see also \cite{TY}). 
There are  112 (conjugacy classes of) 
subgroups  $G\subset W(\mathsf D_5)$ which
give rise to fixed-point free actions. 

We focus on the linearizability problem. 
Note that the $\textbf{(RY)}$ invariant (see Section~\ref{sect:gen}) does not provide any information: $W(\mathsf D_5)$ does not contain abelian subgroups of rank 3 that could give a nontrivial obstruction.

We obtain 33 $W(\mathsf D_5)$-conjugacy classes of subgroups satisfying our assumptions on fixed points; several of these are conjugated in $\PGL_5$. 
We list the remaining cases:

\

{\tiny
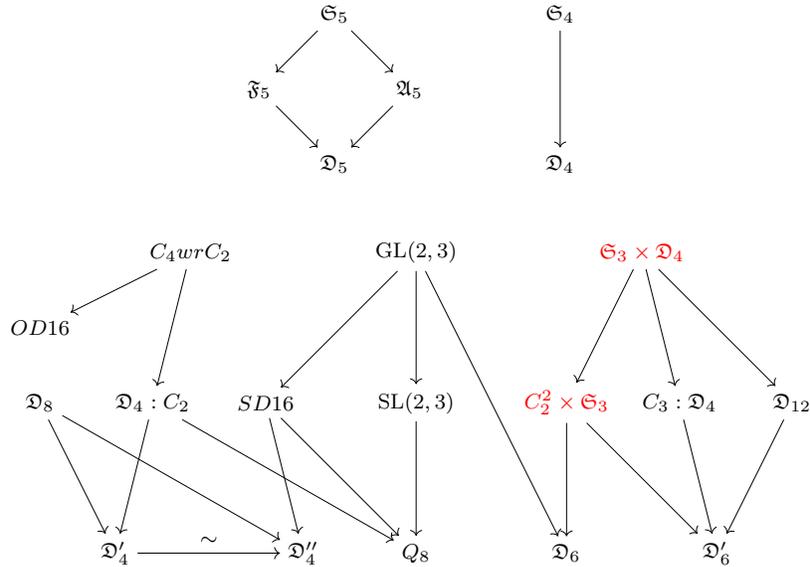
\[
\begin{tikzpicture}[commutative diagrams/every diagram]
	\node (P1) at (-1,-1) {$\mathfrak{D}_5$};
	\node (P15) at (-2,0)
	{$\mathfrak{F}_5$};
	\node (P29) at (0,0)
	{$\mathfrak{A}_5$};
	\node (P33) at (-1,1)
	{$\mathfrak{S}_5$};
	\node (P4) at (2,-1)
	{$\mathfrak{D}_4$};
	\node (P20) at (2,1)
	{$\mathfrak{S}_4$};
	\path[commutative diagrams/every arrow, every label]
	(P33) edge node {} (P15)
	(P33) edge node {} (P29)
	(P15) edge node {} (P1)
	(P29) edge node {} (P1)
	(P20) edge node {} (P4);
\end{tikzpicture}
\]
}

{\tiny
\[
\begin{tikzpicture}[commutative diagrams/every diagram]
	\node (P3) at (1,-2)
        {$\mathfrak{D}_4'$};
	\node (P7) at (3.5,-2)
	{$\mathfrak{D}_4''$};
	\node (P2) at (5,-2)
	{$Q_8$};
	\node (P19) at (0,0)
	{$\mathfrak{D}_8$};
	\node (P16) at (1.5,0)
	{$\mathfrak{D}_4:C_2$};
	\node (P18) at (3,0)
	{$SD16$};
	\node (P24) at (5,0)
	{$\SL(2,3)$};
	\node (P17) at (0,1)
	{$OD16$};
	\node (P30) at (2,2)
	{$C_4 wr C_2$};
	\node (P32) at (5,2)
	{$\GL(2,3)$};
	
	\node (P11) at (7,-2)
	{$\mathfrak{D}_6$};
	\node (P12) at (9,-2)
	{$\mathfrak{D}_6'$};
	\node (P25) at (7,0)
	{{\color{red}$C_2^2\times\mathfrak{S}_3$}};
	\node (P27) at (8.5,0)
	{$C_3:\mathfrak{D}_4$};
	\node (P26) at (10,0)
	{$\mathfrak{D}_{12}$};
	\node (P31) at (8,2)
	{{\color{red}$\mathfrak{S}_3 \times \mathfrak{D}_4$}};
	\path[commutative diagrams/every arrow, every label]
        (P3) edge [above] node {$\sim$} (P7)
	(P30) edge node {} (P17)
	(P30) edge node {} (P16)
	(P32) edge node {} (P18)
	(P32) edge node {} (P24)
	(P32) edge node {} (P11)
	(P19) edge node {} (P3)
	(P19) edge node {} (P7)
	(P16) edge node {} (P3)
	(P16) edge node {} (P2)
	(P18) edge node {} (P7)
	(P18) edge node {} (P2)
	(P24) edge node {} (P2)
	(P31) edge node {} (P25)
	(P31) edge node {} (P27)
	(P31) edge node {} (P26)
	(P25) edge node {} (P11)
    (P25) edge node {} (P12)
	(P27) edge node {} (P12)
	(P26) edge node {} (P12);
	
\end{tikzpicture}
\]
}

\

\begin{exam}
\label{exam:quad}
We consider $G=C_2^2\times \fS_3\subset W(\mathsf D_5)$. The action is realized via involutions $c_4$ and $c_5$ switching signs on $x_4$ and $x_5$, and the permutation action by $\fS_3$ on the remaining variables $x_1,\ldots, x_3$. 

This contributes the symbol 
$$
\bar{\mathfrak s}:= (\bar{H}, \bar{Y} \actsfromleft k(Q), (1)) \in \Burn_3(G),
$$
to the class $[X\actsfromright G]$; 
here $\bar{H}:=\langle c_5\rangle$, and 
$\bar{Y}:=\langle c_4, \fS_3\rangle \simeq C_2\times \fS_3$ is acting on the quadric surface $Q\subset \bP^3$, given by
\begin{align}
\label{eqn:quad-surf}
    \sum_{i=1}^4x_i^2=0. 
\end{align}
We claim that 
\begin{itemize}
\item[(1)] $\bar{\mathfrak s}$ is an incompressible divisorial symbol in $\Burn_3(G)$,
\item[(2)] the $\bar{Y}$ action on $Q$ is not birational to a (projectively) linear action, or products of such actions. 
\end{itemize}

We have addressed (1) in Example~\ref{exam:qqq}. 
The same argument shows 
that the $\bar{Y}$-action on 
$Q$ is not (projectively) linearizable.  
Note also that in this case, we do not need to pass to a standard model $\tilde{X}$ for the $G$-action. Indeed, 
when the class is computed on $\tilde{X}$, it will be a sum 
of various classes,
with {\em positive} coefficients, 
and the incompressible class $\bar{\mathfrak s}$ will be among them.
Since symbols $\bar{\mathfrak s}$ are not produced by the algorithm in Section~\ref{sect:com-lin} and since $\bar{\mathfrak s}$ is incompressible, we conclude that the $G$-action on $X$ is not (projectively) linearizable. 

This $G$ is contained in 
$\fS_3\times \fD_4$, so that the corresponding action on $X$ 
is therefore also not (projectively) linearizable. 
\end{exam}

\begin{exam}
Consider the quadric threefold $X$ given by
$$
\sum_{i=1}^6x_i^2 = \sum_{i=1}^6x_i =0.
$$
It carries a natural action of $\fS_6$, by permutation of the coordinates as well as the induced action of $\fA_6$. By \cite[Theorem 6.2]{ChS-5}, the $\fA_6$-action is super-rigid, in particular, it is not equivariantly birational to a projectively linear action. 

Here we give an alternative argument, based on the Burnside formalism. First we treat $G=\fS_6$. The involution $x_5\leftrightarrow x_6$ fixes a quadric surface $Q$ with residual $\fS_4$-action. We have: 
\begin{itemize}
\item The corresponding symbol
$$
\bar{\mathfrak s} :=(C_2, \fS_4\actsfromleft k(Q), (1))
$$
is incompressible. Indeed, symbols appearing in the $\Theta_2$-term 
actions on the projectivization of a rank-2 vector bundle over $\bP^1$. Since $\fS_4$ does not have normal cyclic subgroups, it has to act trivially on the fibers, and generically freely on the base $\bP^1$. 
In particular, any $\fK_4 \subset \fS_4$ would act without fixed points. 
On the other hand, the $\fK_4$-action on $Q$, generated by the transpositions $(1,2)$ and $(3,4)$, switching $x_1,x_2$ and $x_3,x_4$, respectively, 
fixes two points. 
This implies that 
$\bar{\mathfrak s}$ is incompressible.
\item 
There are two {\em projectively linear} $\fS_6$-actions on $\bP^3$, with 
Burnside classes presented in Section~\ref{sect:3}. 
The symbol $\bar{\mathfrak s}$ does not appear in these expressions. 
\end{itemize}
We conclude that the $\fS_6$-action on $X$ is not birational to a projectively linear action on $\bP^3$. 

Now we give a different argument, for $G:=\fA_6$, and by extension $\fS_6$. Here, we base the argument on computations in 
\begin{align*}
    \BC_3(\fA_6)=\bZ/2 \oplus \bZ. 
\end{align*}
We analyze the fixed loci for (conjugacy  classes of) subgroups $H\subset G$: 

\

\begin{center}
\begin{tabular}{|c|c|c|}
    \hline
     stabilizer $H$& $Z_G(H)$ &orbit representatives of fixed loci of $H$ \\
     \hline
     $\fA_4$ & $1$ &one point\\
    \hline
     $\fA_4'$&$1$&one point\\
     \hline
    $\fS_3$&$1$&two points    \\
    \hline
    $C_3^2$&$C_3^2$&one point\\
    \hline
    $C_5$&$C_5$&two points\\
    \hline
    $C_4$&$C_4$&two points\\
    \hline
    $C_3$&$C_3^2$&one conic\\
    \hline
    $C_3'$&$C_3^2$&one line\\
    \hline
     $C_2$&$\fD_4$&one conic\\
    \hline
\end{tabular}   
\end{center}

\

Note that all symbols in  $\BC_3(\fA_6)$ with stabilizer not equal to 
$H:=C_3^2$ are trivial. The group $H=\langle (1,2,3),(4,5,6)\rangle$ has four fixed points, contained in the $G$-orbit of
$$
\mp=(0:0:0:1:\zeta_3:\zeta_3^2).
$$
The $G$-action is not in standard form; however, since 
$H=C_3^2$ is maximal, in the poset of groups with nontrivial fixed loci, symbols with this stabilizer on a standard form can only arise from blowing up these fixed points. Relation $(\mathbf{B})$ implies that contributions from $H$-fixed points on the blowup equal to those on $X$. 
Thus 
$$
[X\actsfromright G]=(H, 1, ((0,2), (1,2),(2,2))) \in \BC_3(G),
$$
which vanishes, by relation $(\mathbf{V})$. On the other hand, the classes of projectively linear actions of $G$ do not vanish in $\BC_3(G)$, see Section~\ref{sect:3}.
\end{exam}

\bibliographystyle{plain}
\bibliography{gquadric}
\end{document}